\documentclass[11pt,reqno]{amsart} % last updated by JMU on 6/Out/2022

\usepackage{lineno,hyperref, color}
\usepackage{amssymb,amsmath,amsthm,latexsym,bbm}
\usepackage{comment}

\newtheorem{teo}{Theorem}
\newtheorem{lema}{Lemma}

\newtheorem{rem}{Remark}

\newtheorem{cor}{Corollary}
\newtheorem{definition}{Definition}
\newtheorem{proposition}{Proposition}

\title[Nonlocal equations with deforming kernels]{Cordes-Nirenberg type results for nonlocal equations with deforming kernels}

\author[D. dos Prazeres]{Disson dos Prazeres}
\address{Universidade Federal de Sergipe, Brazil}{}
\email{disson@mat.ufs.br}

\author[A. Sobral]{Aelson Sobral}
\address{Department of Mathematics, Universidade Federal da Pa\-ra\'\i \-ba, 58059-900 Jo\~ao Pessoa, PB-Brazil}{}
\email{aelson.sobral@academico.ufpb.br}

\author[J.M.~Urbano]{Jos\'{e} Miguel Urbano}
\address{King Abdullah University of Science and Technology (KAUST), Computer, Electrical and Mathematical Sciences and Engineering Division (CEMSE), Thuwal 23955-6900, Saudi Arabia and University of Coimbra, CMUC, Department of Mathematics, 3001-501 Coimbra, Portugal}{} 
\email{miguel.urbano@kaust.edu.sa}

\begin{document}


%35J60  	Nonlinear elliptic equations
%
%35J96  	Monge-Ampère equations
%
%45H05  	Integral equations with miscellaneous special kernels
%
%35B65  	Smoothness and regularity of solutions to PDEs
%
%35D40  	Viscosity solutions to PDEs
%


\begin{abstract}
We derive Cordes-Nirenberg type results for nonlocal elliptic integro-dif\-ferential equations with deforming kernels comparable to sections of a convex solution of a Monge-Amp\`ere equation. Under a  natural integrability assumption on the Monge-Amp\`ere solution, we prove a stability lemma allowing the ellipticity class to vary. Using a compactness method, we then derive H\"older regularity estimates for the gradient of the solutions.
\end{abstract}

\date{\today}

\maketitle

\section{Introduction}

The fully nonlinear elliptic integro-differential Isaacs equation 
\begin{equation}\label{MAINEQ}
  \mathcal{I}[u(x),x] : = \inf_{\alpha} \sup_{\beta}  L_{\alpha \beta} u(x) = f (x), \quad x \in \Omega,
\end{equation}
where
$$L_{\alpha \beta} u(x) : =  \int_{\mathbb{R}^n} \left( u(x+y) + u(x-y) - 2u(x) \right) \mathcal{K}_{\alpha, \beta} (x,y) dy,$$
appears naturally in competitive stochastic games when two players are allowed to choose from different strategies at every step in order to maximize the expected value $u(x)$ at the first exit point of the domain $\Omega \subset \mathbb{R}^n$. For arbitrary index sets $\mathcal{A}$ and $\mathcal{B}$, the kernels $\{ \mathcal{K}_{\alpha,\beta} (x,y) \}_{\alpha \in \mathcal{A},\beta \in \mathcal{B}}$ are nonnegative functions measuring the frequency of jumps in the $y$ direction at the point $x$. Further applications and motivations related to nonlocal equations can be found, for example, in \cite{bisci2016,bucur2016,nezza2012}.

The large applicability of this type of models reinforces the relevance of seeking qualitative properties for their solutions. In this paper, we obtain Cordes-Nirenberg type results for solutions of \eqref{MAINEQ}, with $f \in L^\infty (\Omega)$. We argue through an approximation method, associating (\ref{MAINEQ}) to a constant coefficients equation. We are particularly interested in the case where the level sets of the kernels $\mathcal{K}_{\alpha,\beta}$ are comparable to sections of a convex solution $\phi$ of a Monge-Amp\`ere equation, as in \cite{caffarelli2018fully}. Setting $v_\phi:\mathbb{R}^n \rightarrow \mathbb{R}$ as
$$
    v_\phi(y) := \phi(y) - \phi(0) - D\phi(0) \cdot y,
$$
we note that, due to convexity, $v_\phi \geq 0$, and define the sections $S_r^\phi$ of $\phi$ as 
$$S_r^\phi := \left\{  v_\phi < r^2 \right\}.$$

We will study equation (\ref{MAINEQ}) for kernels of the particular form
\begin{equation}\label{KERNELSOFTHEEQ}
   \mathcal{K}_{\alpha,\beta} (x,y) = (2-\sigma) \frac{b_{\alpha,\beta} (x,y)}{{v_\phi }^{\frac{n+\sigma}{2}}(y)}, \quad \mbox{with} \quad b_{\alpha,\beta}:\Omega \times \mathbb{R}^n \rightarrow \left[ {\Lambda}^{-1},\Lambda \right] 
\end{equation}
for constants $\Lambda \geq 1 $ and $\sigma \in (1,2)$. Our main result states that if
\begin{equation}\label{CLOSENESS}
    \sup_{\substack{\alpha \in \mathcal{A}, \beta \in \mathcal{B} \\ (x,y) \in \Omega \times \mathbb{R}^n}} \left|b_{\alpha,\beta}(x,y) - b(y) \right| < \eta,
\end{equation}
for a small enough $\eta$ and $b \in L^\infty (\mathbb{R}^n)$, then solutions of (\ref{MAINEQ}) are of class $\mathcal{C}^{1,\alpha_0}_{loc}(\Omega)$, where $\alpha_0$ is to be specified further. 

The regularity theory for fully nonlinear nonlocal elliptic equations was initiated by Caffarelli and Silvestre in their seminal paper \cite{caffarelli2009regularity}. They treat the case corresponding to $\phi = |.|^2$ and a translation-invariant operator, and develop a nonlocal version of the  ABP estimate, establish a Harnack inequality and derive H\"older regularity estimates. Additional regularity assumptions on the kernels lead to H\"older gradient regularity through a standard cut-off and integration by parts technique. This theory was the starting point for many other developments in the nonlocal setting. For example, an anisotropic scenario was considered  in \cite{CLU,SL} and the case of a general convex $\phi \in \mathcal{C}^2(\mathbb{R}^n)$ satisfying
\begin{equation}\label{MONGEAMPEREEQ}
     \gamma \leq \det(D^2 \phi) \leq \Gamma  \quad \mbox{in $\mathbb{R}^n$},
\end{equation}
for $0<\gamma \leq \Gamma<\infty$, was the object of \cite{caffarelli2018fully}. In this case, since the sections of $\phi$ are comparable to ellipsoids (see Fritz John's lemma, in \cite{gutierrez2000geometric}), the kernels can be very degenerate and improved estimates that take into account the deformation of the sections of $\phi$, which is driven by the Monge-Amp\`ere equation, are required. 

As in the local case, a natural extension is to develop an approximation theory that allows for the inclusion of coefficients in the operator. In general, dealing with $x-$dependant equations can be rather challenging since the dependence could break off the effect of the operator. The case $\phi = |.|^2$ was again dealt with by Caffarelli and Silvestre, who were able to prove in \cite{caffarelli2011regularity} that if two nonlocal operators are close in an appropriate sense, then the solutions of the corresponding equations have the same regularity. To prove H\"older gradient regularity for the solutions, however, the kernels of the constant coefficient operators needed to have a fair amount of regularity, a restriction that would be removed by Kriventsov in \cite{kriventsov2013c} by means of an approximation argument. Putting it bluntly, our paper is to \cite{caffarelli2018fully} what \cite{caffarelli2011regularity} was to \cite{caffarelli2009regularity}. Alternative regularity results are also available in different scenarios and we steer the interested reader to, \textit{e.g.}, \cite{BL,CJ,jin2016schauder,S,serra2015c}.

The heart of the matter in proving gradient H\"older regularity in \cite{caffarelli2011regularity} is to  understand what kind of equation functions of the form
\begin{equation}\label{SCALING1ALFA}
     w_\lambda (x) =  \frac{1}{\lambda^{1+\alpha}}[u-l](\lambda x)
\end{equation}
solve, for an affine function $l$, and capture their behaviour as $\lambda$ approaches zero. The first issue is to somehow force a fair scaling behaviour of the operator, so that (\ref{SCALING1ALFA}) still satisfies an elliptic equation in the same class as that of the equation satisfied by $u$. The second issue is purely nonlocal: in order to apply compactness arguments, one needs to extend previous H\"older regularity results for solutions with certain growth at infinity since this behaviour is always expected. Therefore, finding suitable weights that control the growth of the solutions is crucial. The weights must be comparable to the kernels at infinity, so that we can assure the operator is well-defined at the solutions. In \cite{caffarelli2011regularity}, since the order of the equation is greater than $1$, the weights have the form
$$
   \mathcal{W}(y):= \frac{1}{1 + |y|^{n+\sigma}},
$$
and so the solutions could have growth comparable to $\zeta_{\alpha_c} = |.|^{1+\alpha_c}$, for any $\alpha_c$ satisfying $\alpha_c  < \sigma - 1 $. This is due to the fact that
$$
    \left\| \zeta_{\alpha_c} \right\|_{L^1(\mathbb{R}^n,\mathcal{W})} := \int_{\mathbb{R}^n}|y|^{1+\alpha_c}\mathcal{W}(y)dy < \infty.
$$

The main novelty in our problem, with a general convex $\phi \in \mathcal{C}^2(\mathbb{R}^n)$ satisfying (\ref{MONGEAMPEREEQ}), is that  the setting changes in both issues described above. Foremost, scalings of the form (\ref{SCALING1ALFA}) still satisfy an elliptic equation, but in a different class as the original one. More precisely, we need to scale the kernels (\ref{CLOSENESS}) as
$$
    \frac{b_{\alpha, \beta}(\lambda-,\lambda-)}{v_{\phi_{\lambda}}^{\frac{n+\sigma}{2}}}, \quad \mbox{where} \ \phi_\lambda = \lambda^{-2}\phi(\lambda-),
$$
and it is crucial that, although the ellipticity class changes, the $\phi_\lambda$ still satisfy (\ref{MONGEAMPEREEQ}), with the same bounds. Secondly, since growth at infinity is unavoidable, one needs to beware the choice of the weights. The reasonable alternative is to choose
$$
    \mathcal{W}_{\phi}(y) := \frac{1}{1 + v_{\phi}(y)^{\frac{n+\sigma}{2}}},
$$
since it controls the behaviour of (\ref{KERNELSOFTHEEQ}) at infinity, and so, for some $\alpha_1< \sigma-1$, we need
\begin{equation}\label{INTCONDFORPHI}
    \left\|\zeta_{\alpha_1}\right\|_{L^1(\mathbb{R}^n,\mathcal{W}_\phi)} < \infty.
\end{equation}
If scalings were not in use, the basic setting would be then settled to develop the theory. However, because scaling is imperative in our case, we need to choose scaled weights $\mathcal{W}_{\phi_\lambda}$ and find an exponent $\alpha_\lambda < \sigma - 1$ such that
\begin{equation}\label{INTCONDFORPHILAMBDA}
    \left\|\zeta_{\alpha_\lambda}\right\|_{L^1(\mathbb{R}^n,\mathcal{W}_{\phi_\lambda})} < \infty.
\end{equation}  
We stress that all these ingredients need to be uniform in $\lambda$, and so a cautious analysis needs to be performed so that estimates do not degenerate as $\lambda$ goes to zero.

Unlike in \cite{caffarelli2011regularity}, a regularizing effect on the kernels of the equation occurs through the scaling procedure. More precisely, the following convergence holds true (see proposition \ref{LOCALLYELLIP}), locally uniformly in $\mathbb{R}^n$: 
$$
v_{\phi_\lambda}(y) \xrightarrow{\ \lambda \rightarrow 0 \ }  D^2\phi(0)y \cdot y.
$$
Furthermore, if condition (\ref{INTCONDFORPHI}) is true, then (\ref{INTCONDFORPHILAMBDA}) becomes stable and we are able to prove a stability result that allows the kernels to vary (see Lemma \ref{RA5}). Now, the fact that the family $\{\phi_\lambda \}_{\lambda \in (0,1]}$ satisfies (\ref{MONGEAMPEREEQ}) becomes convenient again, since it makes H\"older estimates from \cite{caffarelli2018fully} uniform in the parameter $\lambda$. Therefore, we make way to import the regularity theory valid for a nice equation with kernels comparable, up to a rotation, to that of the fractional Laplacian, and we may use Kriventsov's  results from \cite{kriventsov2013c}. We prove, as a consequence, $\mathcal{C}^{1,\alpha_0}$ regularity estimates for solutions of (\ref{MAINEQ}) with kernels satisfying only (\ref{KERNELSOFTHEEQ}), (\ref{CLOSENESS}) and (\ref{INTCONDFORPHI}), for
$$
    \alpha_0 < \min\{\alpha_1,\alpha_*\},
$$
where $\alpha_*$ is the exponent obtained in \cite{kriventsov2013c} and $\alpha_1$ is from (\ref{INTCONDFORPHI}). In particular, if our equation does not have $x-$dependence, we slightly improve  the gradient regularity from \cite{caffarelli2018fully} since we demand only (\ref{INTCONDFORPHI}) for the kernels (this is automatically true in the setting $\phi = |.|^2$).

Moreover, as an additional technical difficult, we do not assume any sort of symmetry on the kernels. The aftermath is that many computations in the paper, such as in Proposition \ref{HIP24} and Lemma \ref{RA5}, 
need to be cautiously done. Condition (\ref{INTCONDFORPHI}) is once again instrumental in assuring the estimates are steady with respect to the parameter $\lambda$. It is also important to mention that all of our estimates are uniform with respect to the parameter $\sigma$, which satisfies $\sigma \geq \sigma_0 > 1 + \alpha_1$, with $\alpha_1$ from (\ref{INTCONDFORPHI}).

The paper is organized as follows: in section $\ref{SECTION2}$, after some preliminaries, we introduce the definitions to be used in the remainder of the paper and our main assumptions. In section $\ref{SECTION3}$, we present the stability result and the approximation lemma. By applying an iteration procedure, we derive, in section $\ref{SECTION4}$, the gradient regularity estimates for solutions of $(\ref{MAINEQ})$.   

\section{Preliminaries, definitions and main assumptions}\label{SECTION2}

We gather in this section some definitions and auxiliary results, in addition to stating our main assumption. We start with a simple lemma that will be instrumental in the sequel.

\begin{lema}\label{PREM1}
Let $\Psi_1,\Psi_2 \in \mathcal{C}^2(\mathbb{R}^n)$ be two convex functions such that
$$
 \lambda_1 < det(D^2 \Psi_1), \quad \lambda_2 > det(D^2 \Psi_2),
$$
for $0 < \lambda_1, \lambda_2 < \infty$ and define $v_{\Psi_i}(z) := \Psi_i(z) - \Psi_i(0) - D \Psi_i(0) \cdot z$ for $i = 1,2$. For every $r>0$, there exists a constant $C=C(r,\Psi_1,\Psi_2) $ such that
$$
v_{\Psi_1}(z)  \leq C v_{\Psi_2}(z), \quad \forall z \in \overline{B}_r.
$$
\end{lema}

\begin{proof}
Given $r>0$, consider the $C^2$ auxiliary function $f: \mathbb{R}^n \rightarrow \mathbb{R}$ defined by
 $$
   f(z):= v_{\Psi_1}(z) - Cv_{\Psi_2}(z),
 $$
 where 
 $$
    C = \max_{z \in \overline{B}_r} \left\{\frac{\lambda_{max}(D^2 \Psi_1 (z))}{\lambda_{min}(D^2 \Psi_2(z))} \right\} \geq \left(\frac{\lambda_1}{\lambda_2} \right)^{1/n}.
$$
Then
 $$
 \begin{array}{lll}
    D^2 f(z)  & = & D^2 v_{\Psi_1}(z) - C D^2 v_{\Psi_2}(z) \\
     & & \\
      & = & D^2 \Psi_1(z) - C D^2 \Psi_2(z)\\
      & & \\
      & \leq & (\lambda_{max}(D^2 \Psi_1 (z)) - C\lambda_{min}(D^2 \Psi_2(z)))I_d \leq 0
 \end{array}
 $$
so $f$ is a concave function in $\overline{B}_r$ and, therefore, it stays below  any tangent hyperplane. In particular,
$$
    f(z) \leq f(0) + Df(0)\cdot z = 0, \quad \forall z \in \overline{B}_r.
$$
\end{proof}
We next define the appropriate notion of solution to our problem. 

\begin{definition}[Viscosity solution]\label{VS}
Let $f$ be a bounded and continuous function in $\mathbb{R}^n$. We say a function $u:\mathbb{R}^n \rightarrow \mathbb{R}$, upper semicontinuous in $\overline{\Omega}$,  is a viscosity subsolution  in $\Omega$ of the equation 
$$\mathcal{I}[u(x),x]= f,$$
and we write $\mathcal{I}[u(x),x] \geq f$, if, whenever $x_0 \in \Omega$, $N \subset \Omega$ is a neighbourhood of $x_0$, and $\varphi \in C^2(\overline{N})$ satisfies
$$
\varphi(x_0) = u(x_0) \qquad \textrm{and} \qquad \varphi(y) > u(y), \quad \forall y \in N\backslash\{x_0\},
$$
then, if we choose as test function $\tau$ such that
\[\tau:= \left\{ \begin{array}{lcl}
\varphi & \mbox{in} & N\\
u & \mbox{in} & \mathbb{R}^n \backslash N,
\end{array} \right. \]
we have $\mathcal{I}[\tau(x),x] (x_0) \geq f(x_0)$.

A supersolution is defined analogously and a solution is a function that is both a viscosity subsolution and a viscosity supersolution.
\end{definition}

The definition of the extremal operators is basically the same as is \cite{caffarelli2018fully}. These operators are important to define the uniform ellipticity condition with respect to some class $\mathcal{L}$. Here, $\mathcal{L}$ is a set of linear operators $L$ of the form
$$
    L[u](x) = \int_{\mathbb{R}^n}(u(x+y) + u(x-y) -2u(x))\mathcal{K}(y)dy,
$$
for some kernel $\mathcal{K}$. 

We say that a nonlocal operator $\mathcal{I}$ is elliptic with respect to some class $\mathcal{L}$ if
\begin{equation}
	M^-_{\mathcal{L}}[w](x) \leq \mathcal{I}[(u+w)(x),x] (x) - \mathcal{I}[u(x),x] (x) \leq M^+_{\mathcal{L}}[w](x),
\end{equation}
where
$$
	M^-_{\mathcal{L}}[w](x):= \inf_{L \in \mathcal{L}}L[w](x), \quad M^+_{\mathcal{L}}[w](x):= \sup_{L \in \mathcal{L}}L[w](x).
$$
In \cite{caffarelli2018fully}, in order to obtain the ABP estimate, a Harnack inequality and H\"older regularity results, the kernels needed to be in the class $\mathcal{L}^0_\phi(\sigma)$, which is the class of linear operators with kernels $\mathcal{K}$ satisfying
\begin{equation}\label{CLASSOFKJMU}
	(2-\sigma)\frac{1/\Lambda }{v_\phi(y)^{\frac{n+\sigma}{2} }}\leq \mathcal{K}(y) \leq (2-\sigma)\frac{\Lambda}{v_\phi(y)^{\frac{n+\sigma}{2}}},
\end{equation}
for some $\Lambda\geq 1$. It is known that for the class defined by $(\ref{CLASSOFKJMU})$ the extremal operators have the simple form
$$
	M^-_{\mathcal{L}^0_\phi(\sigma)}[u](x) = (2-\sigma)\int_{\mathbb{R}^n}\frac{\frac{1}{\Lambda} \delta^+(u,x,y) - \Lambda \delta^-(u,x,y)}{v_\phi(y)^{\frac{n+\sigma}{2}}}dy
$$
and
$$
		M^+_{\mathcal{L}^0_\phi(\sigma)}[u](x) = (2-\sigma)\int_{\mathbb{R}^n}\frac{\Lambda \delta^+(u,x,y) - \frac{1}{\Lambda} \delta^-(u,x,y)}{v_\phi(y)^{\frac{n+\sigma}{2}}}dy,
$$
where $\delta(u,x,y) = u(x+y) + u(x-y) - 2u(x)$, $\delta^{+}(u,x,y)$ denotes its positive part and $\delta^{-}(u,x,y)$ its negative part. We would like to point out that if $\gamma=\Gamma = 1$ in equation $(\ref{MONGEAMPEREEQ})$ then Pogorelov's result (a proof can be found in \cite[Theorem 4.18]{figalli2017monge}) states that $\phi$ is a quadratic polynomial. Then $v_\phi(y) = Ay \cdot y$, for a matrix $A$ satisfying $\det(A) = 1$. Equations with kernels in this form were studied in \cite{CC} in a different setting related to a nonlocal version of the Monge-Amp\`ere equation. The interested reader may also appreciate \cite{caffarelli2016nonlocal}, for a different approach, and the more recent contribution in  \cite{CSC}.

The requirement that $u \in L^1(\mathbb{R}^n,\mathcal{W})$, for some weight $\mathcal{W}$, allows $u$ to have a certain growth at infinity. For example, in the classical case of \cite{caffarelli2011regularity}, for operators with kernels satisfying the inequality $(\ref{CLASSOFKJMU})$ with $\phi = |.|^2$, the natural choice would be
$$
\mathcal{W}(y) = \frac{1}{1 + |y|^{n+\sigma_0}},
$$
simply because this weight controls the tails of the kernels at infinity, and so the integrals are not singular for $\sigma \geq \sigma_0$.\par
Therefore, a natural choice for the weight in our setting is
\begin{equation}\label{WEIGHT}
\mathcal{W}_\phi(y) = \frac{1}{1 + v_\phi(y)^{\frac{n+\sigma_0}{2}}},
\end{equation}
in order to bound the tails of the kernels $\mathcal{K}$ satisfiying $(\ref{CLASSOFKJMU})$. \par
The weight $\mathcal{W}_{\phi}$ satisfies three properties that will play an important role in our analysis. We next state and prove them.

\begin{proposition}\label{HIP23}
	Assume $\mathcal{K}$ is a function that satisfies
	$$
	    \mathcal{K}(y) \leq \Lambda\frac{(2-\sigma)}{v_\phi(y)^{\frac{n+\sigma}{2}}}.
	$$
	Then for each $r>0$, there exists a constant $C_r$ such that
	$$
	\mathcal{K}(y) \leq C_r\mathcal{W}_\phi(y), \quad \mbox{for} \quad y \not \in S_r^\phi.
	$$
\end{proposition}

\begin{proof}
Notice that for a kernel $\mathcal{K}$ satisfying the assumed bound, we have
$$
	\begin{array}{lll}
		\mathcal{K} (y) &\leq&\displaystyle \Lambda\frac{(2-\sigma)}{v_\phi(y)^{\frac{n+\sigma}{2}}}\\
		& & \\
		& = &\displaystyle \Lambda (2-\sigma)\frac{1 + v_\phi(y)^{\frac{n + \sigma_0}{2}}}{v_\phi(y)^{\frac{n+\sigma}{2}}}\mathcal{W}_\phi(y)\\
		& & \\
		& = & \displaystyle \Lambda (2-\sigma)\left(\frac{1}{v_\phi(y)^{\frac{n+\sigma}{2}}} + v_\phi(y)^{\frac{\sigma_0 - \sigma}{2}} \right)\mathcal{W}_\phi(y)\\
		& & \\
		& \leq &\displaystyle \Lambda (2-\sigma_0)\left(\frac{1}{r^{n+\sigma}} + r^{\sigma_0 - \sigma}\right)\mathcal{W}_\phi(y)\\
		& &\\
		& \leq & 2 \Lambda (2-\sigma_0) \max\left\{ \frac{1}{r^{n+2}},\frac{1}{r^{2-\sigma_0}} \right\}\mathcal{W}_\phi(y)\\
		& & \\
		& = & C_r \mathcal{W}_{\phi}(y),
	\end{array}
$$
for every $y \not \in S_r^\phi$, with $C_r = C(\Lambda, \sigma_0,r,n)$.   
\end{proof}

\begin{proposition}\label{HIP25}
The weight $\mathcal{W}_\phi$ is not singular,\textit{ i.e.}, for each $r>0$, there exists a constant $C_r$ such that
	$$
	\sup_{z \in B_r(y)}\mathcal{W}_\phi(z) \leq C_r\mathcal{W}_\phi(y).
	$$
\end{proposition}

\begin{proof}
$$
	\begin{array}{lll}
	\mathcal{W}_\phi(z) & = & \displaystyle \frac{1}{1 + v_\phi(z)^{\frac{n+\sigma_0}{2}}} =  \displaystyle \frac{1 + v_\phi(y)^{\frac{n+\sigma_0}{2}}}{1 + v_\phi(z)^{\frac{n+\sigma_0}{2}}}\mathcal{W}_\phi(y)\\
		 & & \\
		 & = & \displaystyle \frac{1 + (v_\phi(y) - v_\phi(z) + v_\phi(z))^{\frac{n+\sigma_0}{2}}}{1 + v_\phi(z)^{\frac{n+\sigma_0}{2}}}\mathcal{W}_\phi(y)\\
		 & & \\
		 & \leq & \displaystyle \frac{1 + (|v_\phi(y) - v_\phi(z)| + v_\phi(z))^{\frac{n+\sigma_0}{2}}}{1 + v_\phi(z)^{\frac{n+\sigma_0}{2}}}\mathcal{W}_\phi(y)\\
		 & & \\
		 & \leq & \displaystyle \frac{1 + 2^{\frac{n+\sigma_0}{2}}(|v_\phi(y) - v_\phi(z)|^\frac{n+\sigma_0}{2} + v_\phi(z)^\frac{n+\sigma_0}{2})}{1 + v_\phi(z)^{\frac{n+\sigma_0}{2}}}\mathcal{W}_\phi(y)\\
		 & & \\
		 &\leq & \displaystyle 2^{\frac{n + \sigma_0}{2}}\left(1 + \frac{|v_\phi(y) - v_\phi(z)|^\frac{n+\sigma_0}{2}}{1 + v_\phi(z)^{\frac{n+\sigma_0}{2}}}\right)\mathcal{W}_\phi(y)\\
		 & & \\
		 & \leq & C(n,\sigma_0) \left( 1 + \rho(r)^{\frac{n+\sigma_0}{2}} \right) \mathcal{W}_\phi(y)\\
		 & & \\
			 & = & C_r \mathcal{W}_{\phi}(y), 
	\end{array}
$$
for every $z \in B_r(y)$, with $C_r = C(\sigma_0,r,n)$ and where $\rho$ is the modulus of continuity of $v_\phi$. 

\end{proof}

It is important to notice that the modulus of continuity of $v_\phi$ is the same as that of $\phi$, since $D^2\phi = D^2 v_\phi$. Therefore, the constant $C_r$ is the same for any solution $\Psi$ in the Monge-Amp\`ere class
$$
    \gamma < det(D^2 \Psi) < \Gamma.
$$

The next statement guarantees that the test functions in Definition $\ref{VS}$ are suitable for our operators. These computations have already been made in  \cite{caffarelli2018fully} but we need to assure here they are invariant with respect to solutions of the Monge-Amp\`ere equation.

\begin{proposition}\label{HIP24}
Let $g:\mathbb{R}^n \rightarrow \mathbb{R}$ be a function such that  $g \in \mathcal{C}^2(B_r(x))$,
$$
 \left\|g\right\|_{L^1(\mathbb{R}^n, \mathcal{W}_\phi)} \leq M \quad and \qquad  \left| g(y) - l_{D g(x)}(y) \right| \leq M |y-x|^2, \quad y \in B_r(x),
$$
where $l_{D g(x)}(y) := g(x) + Dg(x) \cdot (y-x)$. Then, there exists a constant $C$ such that 
\begin{equation}
|g(x)| \leq CM
\label{dutra}
\end{equation}
and
\begin{equation}
\left| L[g](x) \right| \leq CM, \quad \forall L \in \mathcal{L}^0_{\phi}(\sigma).
\label{truman}
\end{equation}

\end{proposition}

\begin{proof}
Observe that, for $y \in B_r(0)$, we have
$$\begin{array}{cl}
 |\delta(g,x,y)|  \hspace*{-0.2cm} & = |2g(x) - g(x+y) - g(x-y)|\\
&  \\
&  = |2g(x) - g(x+y) - g(x-y) - l_{D g(x)}(x+y) + l_{D g(x)}(x+y)|\\
& \\
&  \leq \left| g(x-y) - l_{D g(x)}(x-y) \right| + \left| g(x+y) - l_{D g(x)}(x+y) \right| \\
& \\
&  \leq 2M|y|^2.
\end{array}$$
Using the reverse triangle inequality and multiplying both sides by the weight $\mathcal{W}_\phi$, we obtain
$$
 2|g(x)| \mathcal{W}_\phi(y)\leq 2M|y|^2\mathcal{W}_\phi(y) + |g(x+y)|\mathcal{W}_\phi(y) + |g(x-y)|\mathcal{W}_\phi(y).
$$
Since $\mathcal{W}_\phi(y) \leq 1$ and $|y| \leq r$, integrating over $B_r(0)$, we reach
$$
    \begin{array}{lll}
        \displaystyle 2|g(x)|\int_{B   _r}\mathcal{W}_\phi(y)dy & \leq &\displaystyle  2Mr^2|B_r| + \int_{B_r} |g(x+y)|\mathcal{W}_\phi(y)dy\\
         & & \\
         & & \displaystyle +\int_{B_r}|g(x-y)|\mathcal{W}_\phi(y)dy.
    \end{array}
$$
We now bound the two integrals on the right-hand side using Proposition \ref{HIP25}. To uniformly bound
$$
    \int_{B_r} \mathcal{W}_\phi(y)dy 
$$
from below (independently of scalings of the form $\lambda^{-2}v_\phi(\lambda x)$), we use Lemma 1 to get 
$$
    1 + v_{\phi} (y)^{\frac{n+\sigma_0}{2}} = 1 + \left[\frac{v_{\phi}(y)}{|y|^2}\right]^{\frac{n+\sigma_0}{2}}|y|^{n+\sigma_0} \leq   C(r,\phi,n),
$$
thus obtaining \eqref{dutra}.

To prove \eqref{truman}, we start by choosing $\eta$, depending on the modulus of continuity of $\phi$, such that $S^\phi_\eta \subset B_r(0)$.
Then, if $L \in \mathcal{L}^0_\phi(\sigma)$, we have
$$
	\begin{array}{lll}
		\displaystyle \left| \int_{S^\phi_\eta}\delta (g,x,y)\mathcal{K}(y)dy \right|& \leq & \displaystyle (2-\sigma)\Lambda M \int_{S^\phi_\eta}|y|^2\frac{1}{v_\phi(y)^{\frac{n+\sigma}{2}}}dy\\
		 & & \\
		 & \leq & \displaystyle (2-\sigma)\Lambda MC \int_{S^\phi_\eta}|y|^2|y|^{-n-\sigma}dy\\
		 & & \\
		 & \leq & \displaystyle (2-\sigma)\Lambda MC \int_{B_r}|y|^2|y|^{-n-\sigma}dy\\
		 & & \\
		 & = & \displaystyle (2-\sigma)\Lambda MC|\partial B_1| \int_{0}^{r}s^2s^{-n-\sigma}s^{n-1}ds\\
		 & & \\
		 & = & \displaystyle C_1(\Lambda,n,C,r) M \\
	\end{array}
$$
The constant $C = C(r, \phi)$ is from Lemma $(\ref{PREM1})$, which turns out to be invariant by scalings of the form $\lambda^{-2}\phi(\lambda x)$.\par 
Now we estimate the integral at infinity. Let us separate it in three parts.
$$
	\begin{array}{lll}
		\displaystyle \int_{\mathbb{R}^n \backslash S^\phi_\eta}|g(x+y)| \mathcal{K}(y)dy & \leq & \displaystyle C_\eta \int_{\mathbb{R}^n \backslash S^\phi_\eta}|g(x+y)|\mathcal{W}_\phi(y)dy\\
		& & \\
		& \leq & \displaystyle C_\eta \int_{\mathbb{R}^n \backslash S^\phi_\eta}|g(x+y)|\sup_{z \in B_{1}(x+y)}\mathcal{W}_\phi(z)dy\\
		& & \\
		& \leq & \displaystyle C_\eta C_1 \int_{\mathbb{R}^n \backslash S^\phi_\eta}|g(x+y)|\mathcal{W}_\phi(x+y)dy\\
		& & \\
		& \leq & C(\eta,n,\sigma_0, \rho(1))\left\|g\right\|_{L^1(\mathbb{R}^n, \mathcal{W}_\phi)}\\
		& & \\
		& \leq & C(\eta,n,\sigma_0, \rho(1))M.

	\end{array}
$$
where $C_1$ is the constant from Proposition $\ref{HIP25}$ and $\rho$ stands for the modulus of continuity of $\phi$. 

We may assume that the kernels are symmetric (but see Remark $\ref{symmetryofkernels}$ at the end of this section) and apply a similar reasoning to get
$$
		 \displaystyle \int_{\mathbb{R}^n \backslash S^\phi_\eta}|g(x-y)|\sup_{z \in B_{1}(x-y)}\mathcal{W}_\phi(z)dy \leq C(\eta,n,\sigma_0,\rho(1))M.
$$

Finally, using \eqref{dutra}, we can estimate
$$
	\begin{array}{lll}
		\displaystyle \int_{\mathbb{R}^n \backslash S^\phi_\eta} 2|g(x)|\mathcal{K}(y)dy & \leq & \displaystyle CM \int_{\mathbb{R}^n \backslash S^\phi_\eta}\frac{1}{v_\phi(y)^{\frac{n+\sigma}{2}}}dy.
	\end{array}
$$
Notice that we can show that the quantity
$$
	\int_{\mathbb{R}^n \backslash 
	S^\phi_\eta}\frac{1}{v_\phi(y)^{\frac{n+\sigma}{2}}}dy
$$
is finite employing the elementary layer-cake formula \cite[Lemma A.36]{figalli2017monge} and the fact that $|S_r^\phi | \approx r^n$ (depending only on the bounds of the Monge-Ampère equation), as shown in \cite{gutierrez2000geometric}. Indeed,
$$
	\begin{array}{rl}
		\displaystyle \int_{\mathbb{R}^n \backslash S^\phi_\eta}\frac{1}{v_\phi(y)^{\frac{n+\sigma}{2}}}dy  = & \displaystyle (n + \sigma)\int_{0}^{\infty} t^{n + \sigma -1}\left|(\mathbb{R}^n \backslash S^\phi_\eta) \cap \left\{\frac{1}{v_\phi(y)^{\frac{1}{2}}} > t\right\}\right| dt\\
		 & \\
		 = & \displaystyle (n + \sigma)\int_{0}^{\infty} t^{n + \sigma -1}\left|(\mathbb{R}^n \backslash S^\phi_\eta) \cap S_{\frac{1}{t}}\right| dt\\
		 & \\
		 = & \displaystyle (n + \sigma)\int_{0}^{\frac{1}{\eta}} t^{n + \sigma -1}\left|(\mathbb{R}^n \backslash S^\phi_\eta) \cap S_{\frac{1}{t}}\right| dt\\
		 & \\
		 \leq & \displaystyle C(n + \sigma)\int_{0}^{\frac{1}{\eta}} t^{n + \sigma -1}\frac{1}{t^n} dt\\
		 & \\
		 = & C\displaystyle \frac{(n + \sigma)}{\sigma}\frac{1}{\eta^{\sigma}}.
	\end{array}
$$
Putting all estimates together, we obtain that
$$
  \left| L[g](x) \right| =   \left|\int_{\mathbb{R}^n}\delta(g,x,y) \mathcal{K}(y)dy \right| \leq CM.
$$

\end{proof}

We need to impose an integrability assumption on the solution $\phi \in \mathcal{C}^2(\mathbb{R}^n)$ of the Monge-Amp\`ere equation. 

\bigskip

\noindent {\bf [A1]} The function $\zeta:\mathbb{R}^n \rightarrow \mathbb{R}$ defined by $\zeta(y) = |y|^{1+\alpha_1}$ satisfies
\begin{equation}
    \zeta \in L^1 \left( \mathbb{R}^n, \mathcal{W}_{\phi} \right),
\label{MAASSUMP2}
\end{equation}
for some $\alpha_1 < \sigma_0 - 1$.

\bigskip

A simple example of a function in $\mathbb{R}^2$ satisfying the assumption is  
$$\phi(y_1,y_2) = y_1^2 +  y_2^2 + \frac{1}{2}\sin(y_1).$$ 

\begin{proposition}\label{INTGRABILITYREMARK}
Let assumption [A1] be in force and let $\phi_\lambda (x)= \lambda^{-2}\phi(\lambda x)$. Then
$$
   \left\|\zeta\right\|_{L^1 \left( \mathbb{R}^n, \mathcal{W}_{\phi_\lambda} \right) } \leq C \left( n,\sigma_0,\alpha_1,\phi,\left\|\zeta\right\|_{L^1 \left( \mathbb{R}^n, \mathcal{W}_{\phi} \right) } \right).
$$
\end{proposition}

\begin{proof}
We have
$$
    \begin{array}{lll}
        \displaystyle  \left\|\zeta\right\|_{L^1(\mathbb{R}^n, \mathcal{W}_{\phi_\lambda})}& =  & \displaystyle \int_{\mathbb{R}^n} |y|^{1+\alpha_1} \frac{1}{1+ (\lambda^{-2}v_{\phi}(\lambda y))^{\frac{n+\sigma_0}{2}}} dy  \\
         &  & \\
         & = & \displaystyle \lambda^{-n-1-\alpha_1}\int_{\mathbb{R}^n} |z|^{1+\alpha_1} \frac{1}{1+ (\lambda^{-2}v_{\phi}(z))^{\frac{n+\sigma_0}{2}}}dz\\
         & & \\
         & = & \lambda^{-n-1-\alpha_1}(A_1 + A_2 + A_3),
    \end{array}
$$
where $A_1, A_2, A_3$ stand for the integrals of $|.|^{1+\alpha_1}(1 + (\lambda^{-2}v_{\phi}(.))^{\frac{n+\sigma_0}{2}})^{-1}$ over the sets $B_\lambda$, $B_1\backslash B_\lambda$ and $\mathbb{R}^n \backslash B_1$, respectively. 

Now,
$$
A_1  \leq \int_{B_\lambda} |z|^{1+\alpha_1}dz \leq C(n)\lambda^{n+1+\alpha_1},
$$
and so $\lambda^{-n-1-\alpha_1}A_1 \leq C(n)$. For $A_2$, we have 
$$
    A_2 \leq \lambda^{n+\sigma_0}\int_{B_1 \backslash B_\lambda} |z|^{1+\alpha_1} \frac{1}{v_{\phi}(z)^{\frac{n+\sigma_0}{2}}}dz.
$$
By Lemma \ref{PREM1}, with $r=1$, $\Psi_1 = |.|^2$ and $\Psi_2 = \phi$, we get a constant $C(\phi)$ such that
$$
    \frac{|z|^2}{v_\phi(z)} \leq C(\phi), \quad \mbox{for $z \in \overline{B}_1$}.
$$
Then
\begin{eqnarray*}
A_2 & \leq & C(\phi,n)\lambda^{n+\sigma_0}\int_{B_1 \backslash B_\lambda}|z|^{1+\alpha_1}\frac{1}{|z|^{n+\sigma_0}}dz\\
&  = & C(\phi,n)\lambda^{n+\sigma_0} \int_{\lambda}^{1}r^{\alpha_1 - \sigma_0}dr \\
& \leq & C(\phi,n)\lambda^{n+\sigma_0}\frac{1}{\sigma_0 - 1 - \alpha_1}[\lambda^{\alpha_1 - \sigma_0 + 1} - 1].
\end{eqnarray*}
Therefore, we get 
$$
    \lambda^{-n-1-\alpha_1}A_2 \leq C(\phi,n,\sigma_0,\alpha_1).
$$
Finally, for $A_3$, we have $$
    \lambda^{-n-1-\alpha_1}A_3 \leq \lambda^{\sigma_0 -1 - \alpha_1}\int_{\mathbb{R}^n \backslash B_1}|z|^{1+\alpha_1}\frac{1}{v_\phi(z)^{\frac{n+\sigma_0}{2}}}dz .
$$
Therefore, if assumption [A1] holds, we can get a uniform bound (in the parameter $\lambda$) for the term $A_3$. 

\end{proof}

\begin{rem}\label{symmetryofkernels}
If the kernels are not symmetric, we control the quantity 
$$
    \int_{\mathbb{R}^n \backslash S^\phi_\eta} |g(x-y)|\frac{1}{v_\phi(y)^{\frac{n+\sigma}{2}}}dy,
$$
using assumption [A1] and a natural growth condition. Indeed, assuming $g:\mathbb{R}^n \rightarrow \mathbb{R}$ is a function such that
$$
       |g(y)| \leq |y|^{1+\alpha_1} \quad \mbox{for  $y \in \mathbb{R}^n \backslash S^\phi_\eta$},
$$
for some $\eta>0$, we obtain

$$\begin{array}{lll}
        \displaystyle \int_{\mathbb{R}^n \backslash S^\phi_\eta} |g(x-y)|\frac{1}{v_\phi(y)^{\frac{n+\sigma}{2}}}dy & \leq & \displaystyle \int_{\mathbb{R}^n \backslash S^\phi_\eta} |x-y|^{1+\alpha_1}\frac{1}{v_\phi(y)^{\frac{n+\sigma}{2}}}dy \\
         & & \\
         & \leq & \displaystyle C(|x|)\int_{\mathbb{R}^n \backslash S^\phi_\eta} |y|^{1+\alpha_1}\frac{1}{v_\phi(y)^{\frac{n+\sigma}{2}}}dy\\
         & & \\
         & \leq & \displaystyle C(\eta,|x|)\int_{\mathbb{R}^n \backslash S^\phi_\eta} |y|^{1+\alpha_1}\mathcal{W}_\phi(y) dy.
    \end{array}$$
In our main theorem, the growth condition stated above will appear naturally and is, by no means, restrictive.
\end{rem}

\section{Approximation results}\label{SECTION3}
In this section we are going to deliver the key lemma of this paper, consisting in showing that if the operators in our class are close in some suitable sense then so are their solutions. 

We start with a standard result that makes use of the ellipticity structure to restrict the set of test functions in Definition \ref{VS}, which is important to simplify the proof of the stability result. The proof is essentially the same as in \cite{caffarelli2011regularity} but we include it here for the reader's convenience.

\begin{lema}\label{RA1}
	In Definition \ref{VS}, it is enough to consider, as test functions $\varphi$, quadratic polynomials and, as neighbourhoods $N,$ balls centered at the touching point.
\end{lema}

\begin{proof}
	Let $\varphi \in C^2(\overline{N})$ be a test function that touches $u$ from above at a point $x_0$ and, for simplicity, assume $x_0 = 0$. Let $P_\epsilon$ be the following polynomial
	$$
	P_\epsilon(x) =  \frac{1}{2} \left( D^2\varphi(0) + \epsilon I_n \right) x \cdot x  +  \nabla \varphi(0) \cdot x   + \varphi(0).
	$$
	From Taylor's expansion, 
	$$
	\varphi(x) = \varphi(0) +  \nabla \varphi(0) \cdot x + \frac{1}{2} D^2\varphi(0)x \cdot x   + r_2(x), \quad \mbox{with} \quad \lim\limits_{x \rightarrow 0} \frac{r_2(x)}{|x|^2} = 0,
	$$
so, given $\overline{\epsilon}>0$, there exists $\mu>0$ such that, if  $|x|<\mu$, then
\begin{equation}\label{INEQLEMMA1}
\left| \varphi(x) - \left( \varphi(0) + \nabla \varphi(0) \cdot x  + \frac{1}{2} D^2\varphi(0)x \cdot x \right) \right| \leq \overline{\epsilon}|x|^2
\end{equation}
which implies
$$
\varphi(x) \leq P_{\overline{\epsilon}}(x), \quad \forall x \in B_\mu.
$$
Then $P_{\overline{\epsilon}} \geq \varphi \geq u$ in a neighbourhood $B_r  \subset N$, with $r < \mu$. 

Define
	\[ \tau_{\varphi}(x) = \left\{ \begin{array}{ll}
	\varphi(x) & \mbox{if $x \in N$}\\
	u(x) & \mbox{if $x \not \in N$}.\end{array} \right. \]
and
	\[ \tau_{\overline{\epsilon}}(x) = \left\{ \begin{array}{ll}
	P_{\overline{\epsilon}}(x) & \mbox{if $x \in B_r$}\\
	u(x) & \mbox{if $x \not \in B_r$}.\end{array} \right. \]
With $\mathbbm{1}$ denoting the indicator function, we have
$$\tau_{\overline{\epsilon}}(x) \leq \tau_\varphi(x) + 2\overline{\epsilon}|x|^2 \mathbbm{1}_{B_r}, \quad \forall  x \in \mathbb{R}^n,$$
which is obvious if $x \notin B_r$; for  $x \in B_r$, it follows from $(\ref{INEQLEMMA1})$ that
	$$
	-\overline{\epsilon} |x|^2 \leq \varphi(x) - \left( \varphi(0) +  \nabla \varphi(0) \cdot x + \frac{1}{2} D^2\varphi(0)x \cdot x \right) \leq \overline{\epsilon} |x|^2.
	$$
So we have,
	$$
	\begin{array}{lll}
	\tau_\varphi(x) & = & \varphi(x)\\
	& \geq & \varphi(0) +  \nabla \varphi(0) \cdot x + \frac{1}{2} D^2\varphi(0)x \cdot x - \overline{\epsilon} |x|^2\\
	& = & P_{\overline{\epsilon}}(x) - 2\overline{\epsilon} |x|^2\\
	& =& \tau_{\overline{\epsilon}}(x) - 2\overline{\epsilon}|x|^2 \mathbbm{1}_{B_r}.
	\end{array}
	$$
Since $P_{\overline{\epsilon}}$ is a quadratic polynomial such that $P_{\overline{\epsilon}} \geq \varphi \geq u$ and $\varphi$ touches $u$ from above at $0$, we have $P_{\overline{\epsilon}}(0) = \varphi(0) = u(0)$. Hence, $P_{\overline{\epsilon}}$ touches $u$ from above at $0$ and so, by the definition of viscosity subsolution, we have
	$$
	\mathcal{I}[\tau_{\overline{\epsilon}}(x),x] (0) \geq f(0).
	$$
	On the other hand, by uniform ellipticity, 
	$$
	\begin{array}{lll}
	\mathcal{I}[\tau_\varphi(x),x] (0) - \mathcal{I}[\tau_{\overline{{\epsilon}}}(x),x] (0)&\geq &-M^+_{\mathcal{L}^0_\phi(\sigma)} \left[ 2 \overline{\epsilon}|x|^2\mathbbm{1}_{B_r} \right] (0)\\
	& = & - \sup_{L \in \mathcal{L}^0_\phi(\sigma)}L \left[ 2\overline{\epsilon}|x|^2\mathbbm{1}_{B_r}\right] (0)\\
	& \geq & -2\overline{\epsilon}C,
	\end{array}
	$$
	where we have used Proposition \ref{HIP24} in the last inequality. Since 
	$$\mathcal{I}[\tau_{\overline{\epsilon}}(x),x] (0) \geq f(0),$$
	we have
	$$
	\mathcal{I}[\tau_\varphi(x),x] (0) \geq \mathcal{I}[\tau_{\overline{\epsilon}}(x),x] (0) - 2C\overline{\epsilon} , \quad \forall \overline{\epsilon}>0.
	$$
	Consequently, $\mathcal{I}[\tau_\varphi(x),x] (0) \geq f(0)$.
\end{proof}

The following is a simple remark concerning the scalings of the solution of the Monge-Amp\`ere equation.

\begin{proposition}\label{LOCALLYELLIP}
	Let $\phi \in \mathcal{C}^2(\mathbb{R}^n)$ satisfy \eqref{MONGEAMPEREEQ} and define $\phi_\lambda  (y) = \lambda^{-2}\phi(\lambda y)$. Then the function $v_{\phi_\lambda} (y)$ converges to $v_{\phi_0}:= D^2\phi(0) y \cdot y$, locally uniformly as $\lambda \rightarrow 0$, where $\phi_0(y)= D^2\phi(0) y \cdot y$.
\end{proposition}

\begin{proof}

By Taylor's expansion,  
$$
\begin{array}{lll}
v_{\phi_\lambda}(y) & = & \phi_\lambda(y) - \phi_\lambda(0) - \nabla \phi_\lambda(0) \cdot y\\
& & \\
& = & \lambda^{-2} \left( \lambda^2 D^2 \phi(0) y \cdot y + r \left( |\lambda {y}|^2 \right) \right) \\
& & \\
& = &  D^2 \phi(0) y \cdot y + \frac{r(|\lambda {y}|^2)}{\lambda^2}, 
\end{array}
$$
for $y \in B_\eta$, with $\eta$ small enough. Therefore, if $\lambda \rightarrow 0$, 
$$v_{\phi_\lambda}({y}) \longrightarrow D^2 \phi(0) y \cdot y ,$$ 
locally uniformly.	
\end{proof}
	
We now prove the stability lemma, that will play a key role in the approximating results. We will consider the simpler linear case just to highlight the main ideas.

\begin{lema}\label{RA5}
    Let $\Omega$ be open and bounded, and let $\phi \in \mathcal{C}^2(\mathbb{R}^n)$ be convex and satisfy \eqref{MONGEAMPEREEQ}. Consider the family of scalings $\left\{v_{\phi_\lambda}\right\}_{\lambda \in (0,1]}$ and assume there exists a sequence 
    $$
    (\lambda_k, u_k, f_k, b_k) \subset \mathbb{R} \times \mathcal{C}(\overline{\Omega}) \cap L^1 \left( \mathbb{R}^n, \mathcal{W}_{\phi_{\lambda_k}} \right) \times L^\infty({\Omega}) \times L^\infty({\Omega} \times \mathbb{R}^n)
    $$
    such that
	\begin{itemize}
	    \item $\mathcal{I}_k[u_k(x),x] := \displaystyle \int_{\mathbb{R}^n}\delta(u_k,x,y)\frac{b_k(x,y)}{v_{\phi_{\lambda_k}}(y)^{\frac{n+\sigma}{2}}}dy = f_k (x)$
	    in the viscosity sense in $\Omega$;
		
		\item $u_k \rightarrow u$ uniformly in $\overline{\Omega}$ and a.e. in $\mathbb{R}^n$;
		
		\item $|u_k (x)| \leq (1 + \zeta(x))$ in $\mathbb{R}^n$, for some $\alpha_1$ such that $\alpha_1 < \sigma - 1 $;
		
		\item $\lambda_k \rightarrow 0$;
		
		\item $f_k \rightarrow f$, locally uniformly in $\Omega$;
		
		\item $b_k \rightarrow b$, uniformly in ${\Omega} \times \mathbb{R}^n$.
	\end{itemize}
	Then,
	$$
	    \mathcal{I}[u(x),x] (x) = \int_{\mathbb{R}^n}\delta(u,x,y)\frac{b(x,y)}{v_{\phi_0}(y)^{\frac{n+\sigma}{2}}}dy = f(x), \quad x \in \Omega,
	$$
    in the viscosity sense, where $\phi_0$ is as in Proposition \ref{LOCALLYELLIP}.
\end{lema}
\begin{proof}
	We do the proof only for supersolutions. To prove that 
	$$\mathcal{I}[u(x),x] (x) \leq f(x),$$
	we need to show that, given $\varphi \in C^2(x)$, touching $u$ from below at $x$, we have $\mathcal{I}[\tau(x),x] (x) \leq f(x)$, where
	$$
	\tau(x) = \left\{\begin{array}{ll}
	\varphi(x) & \mbox{if $x \in N_x$}\\
	u(x) & \mbox{if $x \not \in N_x$},
	\end{array} \right.
	$$
	where $N_x$ is a neighborhood of $x$. 
By lemma $(\ref{RA1})$, it is enough to consider as test functions quadratic polynomials $p$ and neighbourhoods $N_x = B_r(x)$. Since $u_k$ converges uniformly to $u$ in $\overline{\Omega}$, for large values of $k$, we can find $x_k$ and $d_k$ such that $p + d_k$ touches $u_k$ at $x_k$, with $x_k \rightarrow x$ and $d_k \rightarrow 0$, when $k \rightarrow \infty$. Since $\mathcal{I}_k[u_k(x),x] \leq f_k$ in the viscosity sense in $\Omega$, if we define
	$$
	\tau_k(x) = \left\{ \begin{array}{ll}
	p + d_k & \mbox{in $B_r(x)$}\\
	u_k & \mbox{in $\mathbb{R}^n \backslash B_r(x)$},
	\end{array} \right.
	$$
	then $\mathcal{I}_k[\tau_k(x),x] (x_k) \leq f_k(x_k)$. Clearly, $\tau_k \rightarrow \tau$, uniformly in $B_r(x)$. 
	
Now, let $z \in B_{r/4}(x)$. By the triangle inequality,

$$ \left| \mathcal{I}_k[\tau_k(x),x] (z) - \mathcal{I}[\tau(x),x] (z) \right| $$
$$\leq |\mathcal{I}_k[\tau_k(x),x] (z) - \mathcal{I}_k[\tau(x),x] (z) | + |\mathcal{I}_k[\tau(x),x] (z) - \mathcal{I}[\tau(x),x] (z) |.$$
From the ellipticity condition, and denoting, by simplicity, $\phi_{\lambda_k}$ with $\phi_k$, we obtain
$$\begin{array}{lll}
	\mathcal{I}_k[\tau_k(x),x] (z) - \mathcal{I}_k[\tau(x),x] (z) &\leq&\displaystyle  M^+_{\mathcal{L}^0_{\phi_k}(\sigma)}[\tau_k - \tau](z) \\
	&\leq& \displaystyle |M^+_{\mathcal{L}^0_{\phi_k}(\sigma)}[\tau_k - \tau](z)| 
\end{array}$$
and
$$\begin{array}{lll}
 \mathcal{I}_k[\tau(x),x] (z) - \mathcal{I}_k[\tau_k(x),x] (z) &\leq& \displaystyle  M^+_{\mathcal{L}^0_{\phi_k}(\sigma)}[\tau - \tau_k](z) \\
 &\leq& \displaystyle |M^+_{\mathcal{L}^0_{\phi_k}(\sigma)}[\tau - \tau_k](z)|.
\end{array}
$$
Hence,
$$\hspace*{-4cm} |\mathcal{I}_k[\tau_k(x),x] (z) - \mathcal{I}_k[\tau(x),x] (z) |$$
	$$
	\begin{array}{lll}
	  & \leq &\displaystyle  \max\left\{|M^+_{\mathcal{L}^0_{\phi_k}(\sigma)}[\tau_k - \tau](z)|,|M^+_{\mathcal{L}^0_{\phi_k}(\sigma)}[\tau- \tau_k](z)| \right\}\\
	& & \\
	& \leq &\displaystyle  \sup_{L \in {\mathcal{L}^0_{\phi_k}(\sigma)}}|L[\tau_k-\tau](z)|\\
	& & \\ 
	& \leq &  \displaystyle \int_{\mathbb{R}^n} |\delta(\tau_k - \tau,z,y)|\frac{\Lambda (2-\sigma)}{v_{\phi_k}(y)^{\frac{n+\sigma}{2}}}dy.
	\end{array}
	$$
For points $y \in B_{r/2}$, we have $z+y \in B_r(x)$. Therefore, $\tau-\tau_k = d_k$, which implies $\delta(d_k,z,y) = 0$. Plugging this into the above inequality, we get
$$\hspace*{-4cm} |\mathcal{I}_k[\tau_k(x),x] (z) - \mathcal{I}_k[\tau(x),x] (z) |$$
	$$
	\begin{array}{lll}
		& \leq& \displaystyle \int_{\mathbb{R}^n \backslash B_{r/2}}|\delta(\tau_k - \tau,z,y)|\frac{\Lambda (2-\sigma)}{v_{\phi_k}(y)^{\frac{n+\sigma}{2}}}dy \\
		& & \\
		& \leq & \displaystyle\int_{\mathbb{R}^n \backslash S_\mu^{\phi_k}} |(\tau_k - \tau)(z+y)|\frac{\Lambda (2-\sigma)}{v_{\phi_k}(y)^{\frac{n+\sigma}{2}}}dy \\
		& & \\
		&  & + \displaystyle\int_{\mathbb{R}^n \backslash S_\mu^{\phi_k}} |(\tau_k - \tau)(z-y)|\frac{\Lambda (2-\sigma)}{v_{\phi_k}(y)^{\frac{n+\sigma}{2}}}dy \\
		& & \\
		& & + \displaystyle 2|(\tau_k - \tau)(z)|\int_{\mathbb{R}^n \backslash S_\mu^{\phi_k}}\frac{\Lambda (2-\sigma)}{v_{\phi_k}(y)^{\frac{n+\sigma}{2}}}dy,
	\end{array}
	$$
where we choose $\mu = \mu(r)$ such that 
$$S_\mu^{\phi_k} \subset B_{r/2}, \quad \forall k \in \mathbb{N}.$$ 
Now we apply Propositions \ref{HIP23} and \ref{HIP25} to get  
$$\hspace*{-3cm} \int_{\mathbb{R}^n \backslash S^{\phi_k}_\mu}|(\tau_k - \tau)(z+y)|\frac{\Lambda (2-\sigma)}{v_{\phi_k}(y)^{\frac{n+\sigma}{2}}}dy$$
	$$
	\begin{array}{lll}
	 & \leq & \displaystyle C_{\mu}\int_{\mathbb{R}^n \backslash S^{\phi_k}_\mu}|(\tau_k - \tau)(z+y)|\mathcal{W}_{\phi_k}(y)dy\\
	& & \\
	& \leq & \displaystyle C_{\mu} C\int_{\mathbb{R}^n \backslash S^{\phi_k}_\mu}|(\tau_k - \tau)(y+z)|\mathcal{W}_{\phi_k}(y+z)dy \\
	& & \\
	& \leq & C\left\|\tau_k - \tau \right\|_{L^1 \left( \mathbb{R}^n\backslash S_\mu^{\phi_k} ,\mathcal{W}_{\phi_k}\right)}.
	\end{array}
	$$
Furthermore, by computations made in the proof of Proposition (\ref{HIP24}), we obtain
	$$
	\int_{\mathbb{R}^n \backslash S^{\phi_k}_\mu}\frac{\Lambda(2-\sigma)}{v_{\phi_k}(y)^{\frac{n+\sigma}{2}}}dy \leq \frac{n+2}{\sigma_0  \mu^{\sigma_0}} 
	$$
and so
$$\hspace*{-4cm} |\mathcal{I}_k[\tau_k(x),x] (z) - \mathcal{I}_k[\tau(x),x] (z) | $$
\begin{eqnarray*}
& \leq & C_1\left\|\tau_k - \tau \right\|_{L^1 \left( \mathbb{R}^n\backslash S^{\phi_k}_\mu,\mathcal{W}_{\phi_k} \right) } + C_2|\tau_k(z) - \tau(z)| \\
&  & + \displaystyle\int_{\mathbb{R}^n \backslash S_\mu^{\phi_k}} |(\tau_k - \tau)(z-y)|\frac{\Lambda (2-\sigma)}{v_{\phi_k}(y)^{\frac{n+\sigma}{2}}}dy.
\end{eqnarray*}
Since $u_k \rightarrow u$, a.e. in $\mathbb{R}^n$, we get that $\tau_k \rightarrow \tau$, a.e. in $\mathbb{R}^n$. Therefore, by Proposition \ref{INTGRABILITYREMARK}, assumption [A1] and the growth assumption we obtain, by the dominated convergence theorem, that
$$\left\| \tau_k - \tau \right\|_{L^1 \left( \mathbb{R}^n\backslash S_\mu^{\phi_k}, \mathcal{W}_{\phi_k} \right)} \longrightarrow 0.$$
Note that, by the same computations made in Remark \ref{symmetryofkernels}, we get
$$\hspace{-2cm} \int_{\mathbb{R}^n \setminus S^{\phi_k}_\mu} \left| (\tau_k - \tau)(z-y) \right| \frac{\Lambda (2-\sigma)}{v_{\phi_k}(y)^{\frac{n+\sigma}{2}}}dy$$
\begin{eqnarray*}
 & \leq & C \left( \mu, |z|, \Lambda \right) \int_{\mathbb{R}^n \setminus S^{\phi_k}_\mu} |y|^{1+\alpha_1}\mathcal{W}_{\phi_k}(y) dy,\\
& = & C \left( \mu, \Omega,  \Lambda \right) \left\| \zeta \right\|_{L^1\left( \mathbb{R}^n\backslash S_\mu^{\phi_k} ,\mathcal{W}_{\phi_k} \right)}  \\
& \leq & \overline{C},
\end{eqnarray*}
where we used assumption [A1], along with Proposition \ref{INTGRABILITYREMARK}, to get a uniform bound on the weighted norm. 

Now, notice that, for every $z \in B_{r/4}(x)$, we have $\tau \in \mathcal{C}^2(B_{3r/4}(z))$ and thus
$$\mathcal{I}_k[\tau(x),x] (z) \rightarrow \mathcal{I}[\tau(x),x] (z), \quad \mbox{uniformly in $B_{r/4}(x)$ when $k \rightarrow \infty$}.$$
Therefore, we can combine this with Lemma \ref{PREM1}, Proposition \ref{INTGRABILITYREMARK} and assumption [A1] to get 
	$$
	  |\delta(\tau,z,-)|\left|\frac{b_k(z,-)}{v_{\phi_k}(-)^{\frac{n+\sigma}{2}}} - \frac{b(z,-)}{v_{\phi_0}(-)^{\frac{n+\sigma}{2}}} \right| \in L^1(\mathbb{R}^n),
	$$
	uniformly in $z$. Since 
    $$
        \frac{b_k(z,y)}{v_{\phi_k}(y)^{\frac{n+\sigma}{2}}} \longrightarrow  \frac{b(z,y)}{v_{\phi_0}(y)^{\frac{n+\sigma}{2}}}, \quad \mbox{a.e. $\in \mathbb{R}^n$},
    $$
    we get, again using the dominated convergence theorem, that
	$$\left| \mathcal{I}_k[\tau(x),x] (z) - \mathcal{I}[\tau(x),x] (z) \right| $$
	$$\leq \int_{\mathbb{R}^n}|\delta(\tau,z,y)|\left|\frac{b_k(z,y)}{v_{\phi_k}(y)^{\frac{n+\sigma}{2}}} - \frac{b(z,y)}{v_{\phi_0}(y)^{\frac{n+\sigma}{2}}} \right| dy  \longrightarrow 0.
	$$
We thus obtain
	\begin{equation}\label{RA7}
	\mathcal{I}_k[\tau_k(x),x] (z) \rightarrow \mathcal{I}[\tau(x),x] (z),
	\end{equation}
	uniformly in $B_{r/4}(x)$. 
	
	Finally, we have
	\begin{eqnarray*}
	|\mathcal{I}_k[\tau_k(x),x] (x_k) - \mathcal{I}[\tau(x),x] (x) | & \leq & |\mathcal{I}_k[\tau_k(x),x] (x_k) - \mathcal{I}[\tau(x),x] (x_k) | \\
	& & + |\mathcal{I}[\tau(x),x] (x_k) - \mathcal{I}[\tau(x),x] (x) |
\end{eqnarray*}
and, using \eqref{RA7}) and  the continuity of $\mathcal{I}[\tau(x),x]$, we obtain 
	$$
	|\mathcal{I}_k[\tau_k(x),x] (x_k) - \mathcal{I}[\tau(x),x] (x) | \rightarrow 0,
	$$
	when $k \rightarrow \infty$. Since $x_k \rightarrow x$ and $f_k \rightarrow f$ locally uniformly,
	$$
	\mathcal{I}[\tau(x),x] (x) = \lim\limits_{k \rightarrow \infty}\mathcal{I}_k[\tau_k(x),x] (x_k)\leq  \lim\limits_{k \rightarrow \infty}f_k(x_k) = f(x),
	$$
	and therefore $\mathcal{I}[\tau(x),x] (x) \leq f(x)$.
\end{proof}

We are now ready for the approximation lemma. It plays an important role in the subsequent analysis.
\begin{lema}\label{RA4}
Assume $\sigma \geq \sigma_0 > 1+\alpha_1$ and $\phi \in C^2(\mathbb{R}^n)$ is a convex function satisfying \eqref{MONGEAMPEREEQ} and assumption [A1]. Let $\phi_0$ be the function from Proposition \ref{LOCALLYELLIP}. Given $M>0$, a modulus of continuity $\rho$ and $\epsilon>0$, there exist a small $\eta>0$, $\lambda>0$ and a large $R_0>0$ such that, if

\begin{itemize}
\item $-\eta < \displaystyle \inf_{\alpha \in \mathcal{A}} \sup_{\beta \in \mathcal{B}}\int_{\mathbb{R}^n}\delta(w,-,y)\frac{b_{\alpha \beta}(-,y)}{v_{\phi_\lambda}(y)^{\frac{n+\sigma}{2}}}dy < \eta$, 
in the viscosity sense in $S_1^{\phi_0}$;
\medskip
\item for every $\alpha$ and $\beta$, $b(y) - \eta < b_{\alpha \beta}(x,y) < b(y)+\eta$,  $\forall x \in S_1^{\phi_0}$, $\forall y \in \mathbb{R}^n$;
\medskip
\item $|w(y) - w(x)|  \leq  \rho(|y-x|)$, $\forall x,y \in B_{R_0}$;
\medskip
\item $|w(x)|  \leq  M(1+ \zeta(x))$,  $\forall x \in \mathbb{R}^n$,
\end{itemize}
\medskip
then, there exist a function $w_0:\mathbb{R}^n \rightarrow \mathbb{R}$ such that
\medskip
\begin{itemize}
\item $|w_0(x)| \leq M \left(1  + \zeta (x) \right)$, $\forall x \in \mathbb{R}^n$;
\medskip
\item $ \displaystyle \int_{\mathbb{R}^n}\delta(w_0,-,y)\frac{b(y)}{v_{\phi_0}(y)^{\frac{n+\sigma}{2}}}dy = 0$,  in the viscosity sense in $S_1^{\phi_0}$;
\medskip
\item $|w - w_0| < \epsilon$ in $S_1^{\phi_0}$.
\end{itemize}
\end{lema}

\begin{proof}
Suppose, by contradiction, the lemma is false. Then, there exist sequences $R_k,\eta_k,w_k,\lambda_k,b_k$ such that
$$R_k \rightarrow \infty, \quad \eta_k \rightarrow 0, \quad \lambda_k \rightarrow 0 \quad \mbox{and} \quad \quad b_k \rightarrow b \quad \mbox{uniformly in $S_1^{\phi_0} \times \mathbb{R}^n$}$$ 
and
$$-\eta_k < \displaystyle \int_{\mathbb{R}^n}\delta(w_k,x,y)\frac{b_k(x,y)}{v_{\phi_{\lambda_k}}(y)^{\frac{n+\sigma}{2}}}dy < \eta_k $$
in the viscosity sense in $S_1^{\phi_0}$ but 
\begin{equation}\label{contradiction aproximation lemma}
\sup|w_k - w_0| \geq \epsilon \quad \mbox{in $S^{\phi_0}_1$},
\end{equation}
for all $w_0$ solution of 
$$ \int_{\mathbb{R}^n}\delta(w_0,x,y)\frac{b(y)}{v_{\phi_0}(y)^{\frac{n+\sigma}{2}}}dy = 0.$$

By Proposition \ref{LOCALLYELLIP},
$$v_{\phi_{\lambda_k}} (x) \longrightarrow D^2\phi(0)x \cdot x = v_{\phi_0} (x),$$
locally uniformly and thus almost everywhere in $\mathbb{R}^n$. On the other hand, the $w_k$ have a uniform modulus of continuity in $B_{R_k}$, where $R_k \rightarrow \infty$, so, up to subsequences, $w_k\rightarrow \overline{w}$ uniformly on compacts and thus almost everywhere in $\mathbb{R}^n$.  By Lemma \ref{RA5}, we pass to the limit to get 
$$\int_{\mathbb{R}^n}\delta(\overline{w},x,y)\frac{b(y)}{v_{\phi_0}(y)^{\frac{n+\sigma}{2}}}dy = 0 \quad \mbox{in $S_1^{\phi_0}$}.$$
Since  $w_k\rightarrow \overline{w}$ uniformly in $S_1^{\phi_0}$, for $k$ sufficiently large we have a contradiction with (\ref{contradiction aproximation lemma}).
\end{proof}

\section{ Regularity for variable coefficients}\label{SECTION4}

This section is devoted to the proof of the main theorem of the paper. 

\begin{teo}\label{MAINTHM}
Let $\sigma \geq \sigma_0 >1$, let $\phi \in \mathcal{C}^2(\mathbb{R}^n)$ be a convex function satisfying $(\ref{MONGEAMPEREEQ})$ and assumption [A1], and let $u$ be a viscosity solution to $(\ref{MAINEQ})$ in $S^\phi_{2}$. There exists $0< \eta \ll 1$ such that if
$$
    \sup_{\substack{\alpha \in \mathcal{A}, \beta \in \mathcal{B} \\ (x,y) \in S^\phi_{2} \times \mathbb{R}^n}} \left|b_{\alpha,\beta}(x,y) - b(y) \right| < \eta,
$$
for $b\in L^{\infty}(\mathbb{R}^n)$, then $u \in \mathcal{C}^{1,\alpha_0}(S^\phi_1)$ and the following estimate holds
$$
\left\|u\right\|_{\mathcal{C}^{1,\alpha_0} \left( S^\phi_1 \right) } \leq C \left(\left\|u\right\|_{L^\infty \left( \mathbb{R}^n \right) } + \left\|f\right\|_{L^\infty \left( S^\phi_{2} \right) } \right),
$$
for $0<\alpha_0 < \min\{\alpha_* , \alpha_1  \}$, where $\alpha_1$ is from assumption [A1] and $\alpha_*$ is from Theorem \ref{GRADIENTESTIMATESFORSOLUTIONSWITHGROWTH}. The constant $C$ depends only on $n,\Lambda,\sigma_0,\alpha_\ast, \gamma$ and $\Gamma$.
\end{teo}

The idea is to iterate the approximation lemma (Lemma \ref{RA4}) to quotients of the form
$$
    \frac{[u-l](\lambda -)}{\lambda^{1+\alpha_0}},
$$
where $l$ is an affine function and $\lambda>0$ and $\alpha_0$ are to be chosen in the sequel. To assure that Lemma \ref{RA4} can be applied for every small enough $\lambda$, the first step is to grant a modulus of continuity to solutions that grow at infinity. This is done using a standard technique.
\begin{teo}\label{MODULUSOFCONTINUITYGROWTH}
Assume $\sigma>\sigma_0 > 1$ and $w \in \mathcal{C} \left( \overline{S_{2}^\phi} \right)$ satisfies, for $M>0$ and $\zeta$ from [A1], 
$$|w(x)|  \leq  M(1+ \zeta (x)), \quad x \in \mathbb{R}^n,$$
and solves
$$M^+_{\mathcal{L}^0_\phi(\sigma)} w \geq -C_0, \quad M^-_{\mathcal{L}^0_\phi(\sigma)} w \leq C_0 $$
in the viscosity sense in $S_{2}^\phi$. Then, there exists $\alpha > 0$ and a radius $\overline{\rho}$ such that $w \in \mathcal{C}^\alpha \left( S_{\overline{\rho}}^\phi \right)$ and
    $$
    \left\|w\right\|_{\mathcal{C}^{\alpha} \left( S_{\overline{\rho}}^\phi \right)} \leq C^* \left( \sup_{S_{2}^\phi}|w| + C_0 + C\left\|\zeta\right\|_{L^1 \left( \mathbb{R}^n,\mathcal{W}_{\phi} \right)} \right)
    $$
    for some universal constant $C^*$. 
\end{teo}
\begin{proof}
    Define $\overline{w}(x) = \mathbbm{1}_{S^\phi_{2}}(x)w(x)$ and notice that for $x \in S^\phi_{1}$,
    $$
     \begin{array}{lll}
         L\overline{w}(x) & = &\displaystyle \int_{\mathbb{R}^n}\delta(\overline{w},x,y)\mathcal{K}(y)dy \\
          &  &\\
          & = &\displaystyle  Lw(x) + \int_{\mathbb{R}^n}(w(x+y)\mathbbm{1}_{S^\phi_{2}}(x+y) - w(x+y))\mathcal{K}(y)dy\\
          & & \\ 
          & & \displaystyle + \int_{\mathbb{R}^n}(w(x-y)\mathbbm{1}_{S^\phi_{2}}(x-y) - w(x-y))\mathcal{K}(y)dy\\
          & & \\
          & = & \displaystyle Lw(x) - \int_{\mathbb{R}^n\backslash \Omega^+} w(x+y)\mathcal{K}(y)dy - \int_{\mathbb{R}^n\backslash \Omega^-} w(x-y)\mathcal{K}(y)dy,
     \end{array}   
    $$
    where 
    $$
    \Omega^+ = \{y\in \mathbb{R}^n: x+y \in S^\phi_{2} \} \quad \mbox{and} \quad \Omega^- = \{y\in \mathbb{R}^n: x-y \in S^\phi_{2} \}.
    $$
  Now, if $d = dist(\partial S^\phi_{1},\partial S^\phi_{2})$, we have
    $$
        B_d(0) \subset \Omega^+, \Omega^-. 
    $$
    Therefore, from the growth assumption for $w$, we can estimate
    $$
        \begin{array}{lll}
\displaystyle \left|\int_{\mathbb{R}^n\backslash \Omega^+} w(x+y)\mathcal{K}(y)dy\right| & \leq & \displaystyle C \int_{\mathbb{R}^n\backslash B_d} |y|^{1+\alpha_1}\frac{1}{v_\phi(y)^{\frac{n+\sigma}{2}}}dy\\
            & & \\
            & \leq & C(d)\left\|\zeta\right\|_{L^1(\mathbb{R}^n, \mathcal{W}_\phi)}.
        \end{array}
    $$
    The same computations are valid for the other term and we get
    $$
        M^+_{\mathcal{L}^0_\phi(\sigma)}\overline{w}(x) \geq M^+_{\mathcal{L}^0_\phi(\sigma)} w(x) - C\left\|\zeta\right\|_{L^1(\mathbb{R}^n,\mathcal{W}_{\phi})} \geq -(C_0 +C\left\|\zeta\right\|_{L^1(\mathbb{R}^n,\mathcal{W}_{\phi})}),
    $$
    for $x \in S^\phi_{1}$. The same calculations apply to $M^-_{\mathcal{L}^0_\phi(\sigma)}$ leading to 
        $$
        M^-_{\mathcal{L}^0_\phi(\sigma)}\overline{w}(x) \leq  C_0 +C\left\|\zeta\right\|_{L^1(\mathbb{R}^n,\mathcal{W}_{\phi})} ,
    $$
for $x \in S^\phi_{1}$. We now apply a scaled version of \cite[Theorem 6.2]{caffarelli2018fully}, to get the existence of $\alpha \in (0,1)$ and a small radius $\overline{\rho}$, depending on $\lambda,\Lambda,\sigma_0$ and dimension, such that
$$
        \left\|\overline{w}\right\|_{\mathcal{C}^\alpha \left( S^\phi_{ \overline{\rho}} \right) } \leq C^* \left( \sup_{\mathbb{R}^n}|\overline{w}| + C_0 + C \left\|\zeta\right\|_{L^1 \left( \mathbb{R}^n,\mathcal{W}_{\phi} \right) } \right) .
$$
    Since $\overline{w} = w$ in $S^\phi_{2}$, we also have
 $$
        \left\|w\right\|_{\mathcal{C}^\alpha \left( S^\phi_{\overline{\rho}} \right) } \leq C^* \left( \sup_{S^\phi_{2}}|w| + C_0 + C\left\|\zeta\right\|_{L^1 \left( \mathbb{R}^n,\mathcal{W}_{\phi} \right) } \right) .
$$
This constant $C^*$ depends on the constant $\overline{\tau}$ from \cite[Proposition 3.1]{caffarelli2018fully}, the norm of the normalization mapping of the section $S^\phi_{\overline{\gamma}\rho/2},\Lambda,\sigma_0$ and dimension, where $\overline{\gamma}$ stands for the standard engulfing constant and $\rho$ is the constant from \cite[Theorem 5.1]{caffarelli2018fully}.
\end{proof}
\begin{rem}
We will apply the theorem above for the scaled family $\{\phi_\lambda \}$, with $\lambda \in (0,1]$. Since we need uniform estimates, we must assure that the estimate above does not degenerate as $\lambda$ varies in the interval $(0,1]$. From Proposition \ref{INTGRABILITYREMARK} we can uniformly bound  the quantity $\left\|\zeta\right\|_{L^1(\mathbb{R}^n, \mathcal{W}_{\phi_\lambda})}$. On the other hand, the quantity $d_\lambda = dist(\partial S_{1}^{\phi_\lambda},\partial S_{2}^{\phi_\lambda})$ is uniformly bounded above and below away from zero. We can also bound uniformly the norm of the normalization mapping of $S^{\phi_\lambda}_{\overline{\gamma} \rho /2}$ by making use of \cite[Corollary 4.7]{figalli2017monge}.  
\end{rem}

\begin{cor}
Assume $\sigma>\sigma_0 > 1$ and $w \in \mathcal{C} \left( \overline{S^\phi_{2r}} \right)$ satisfies
    $$
    M^+_{\mathcal{L}^0_\phi(\sigma)} w \geq -C_0, \quad M^-_{\mathcal{L}^0_\phi(\sigma)} w \leq C_0 \quad \mbox{in $S^\phi_{2 r}$},
    $$
    with $|w| \leq M(1+ \zeta)$ in $\mathbb{R}^n$. Then, there exists $\alpha > 0$ such that $w \in \mathcal{C}^\alpha \left( S^\phi_{r/2} \right) $ and
\begin{equation}
    \left\|w\right\|_{\mathcal{C}^{\alpha} \left( S^\phi_{r\overline{\rho}} \right) } \leq C^*\left(\sup_{S^\phi_{2 r}}|w| + C_0 + C\left\|\zeta\right\|_{L^1 \left( \mathbb{R}^n,\mathcal{W}_\phi \right)}  \right),
\label{leitao}
\end{equation}
    for some constant $C^*$.
\label{SCALEDHOLDER}
\end{cor}
\begin{proof}
 Let $w^\ast$ be the following auxiliary function
 $$
    w^\ast(x) := \frac{1}{\max\{1,r^{1+\alpha_1},r^\sigma \}} w(rx).
 $$
 Then
     $$
    M^+_{\mathcal{L}^0_{\phi_r}(\sigma)} w^\ast \geq -C_0, \quad M^-_{\mathcal{L}^0_{\phi_r}(\sigma)} w^\ast \leq C_0 \quad \mbox{in $S^{\phi_r}_{2 }$},
    $$
    for $\phi_r (x)= r^{-2}\phi(rx)$ and $|w^\ast| \leq M(1+ \zeta)$ in $\mathbb{R}^n$. By the previous theorem,
    $$
        \left\|w^\ast\right\|_{\mathcal{C}^{\alpha} \left( S^{\phi_r}_{\overline{\rho}} \right) } \leq C^*\left(\sup_{S^{\phi_r}_{2}}|w^\ast| + C_0 + C\left\|\zeta\right\|_{L^1 \left( \mathbb{R}^n,\mathcal{W}_\phi \right) } \right).
    $$
    Rescaling back to $w$, we get \eqref{leitao}, where $C^*$ has the same dependence as in the previous theorem, replacing for the norm of the normalization mapping of the section $S^{\phi_r}_{\overline{\gamma}\rho /2}$.
    
\end{proof}
We also need gradient H\"older regularity estimates for solutions of the limit equation 
\begin{equation} 
\int_{\mathbb{R}^n}\delta(u,x,y)\frac{b(y)}{v_{\phi_0}(y)^{\frac{n+\sigma}{2}}}dy = 0,  \quad x \in S^{\phi_0}_1. 
\label{forreta}
\end{equation}
It is interesting to note that we will not make use of the $\mathcal{C}^{1,\alpha}$ regularity results from \cite{caffarelli2018fully}, since our kernels do not necessarily satisfy the assumptions therein. Instead, due to the regularizing effect of the scalings, we will make use of the results from \cite{kriventsov2013c}, after a suitable change of variables.

\begin{teo}\label{GRADIENTESTIMATESFORSOLUTIONSWITHGROWTH}
    Let $\sigma \geq \sigma_0 > 1$. Assume $v \in \mathcal{C} \left( \overline{S}^{\phi_0}_1 \right) \cap L^1 \left( \mathbb{R}^n,\mathcal{W}_{\phi_0} \right)$ is a viscosity solution of 
    $$
         \int_{\mathbb{R}^n}\delta(v,x,y)\frac{b(y)}{v_{\phi_0}(y)^{\frac{n+\sigma}{2}}} dy = 0 \quad \mbox{in $S^{\phi_0}_1$},
    $$
    where $\mathcal{K}(y) = b(y)/v_{\phi_0}(y)^{\frac{n+\sigma}{2}} \in  \mathcal{L}^0_{\phi_0}(\sigma)$. Then, there exists $\alpha_*$ such that $v \in \mathcal{C}^{1,\alpha_*} \left( S^{\phi_0}_{1/2} \right)$ and
    $$
        \left\|v\right\|_{\mathcal{C}^{1,\alpha_*}(S^{\phi_0}_{1/2})} \leq C \left( \sup_{S^{\phi_0}_1}|v| + \left\|v\right\|_{L^1 \left( \mathbb{R}^n,\mathcal{W}_{\phi_0} \right) } \right),
    $$
where $C>0$ is a universal constant.
\end{teo}
\begin{proof}
We remark that $v_{\phi_0}(y) = D^2\phi(0) y \cdot y$. Since $\phi$ solves $(\ref{MONGEAMPEREEQ})$, we get that, in particular, the matrix $D^2 \phi(0)$ has only positive eigenvalues, and so is invertible. Therefore, it may be decomposed in the following form $D^2\phi(0) = S D S^t$, where $S$ is an orthogonal matrix and $D$ is diagonal with the eigenvalues of $D^2 \phi(0)$. Since the eigenvalues are nonnegative, the matrix $D$ has a square root, which we will denote by $B := \sqrt{D}$. Then
$$
    D^2 \phi(0)y \cdot y = Qy \cdot Qy,
$$
where $Q := BS^t$. Notice that 
$$
    x \in S_1^{\phi_0} \iff D^2\phi(0)x \cdot x < 1 \iff |Qx| < 1 \iff x \in Q^{-1}(B_1),
$$
that is, $Q(S_1^{\phi_0}) = B_1$. Now, define $w = v \circ Q^{-1}$. Since $v \in \mathcal{C}(\overline{S}^{\phi_0}_1)$, we obtain $w \in \mathcal{C}(\overline{B}_1)$. For the growth condition, we have
$$
    \begin{array}{lll}
    \displaystyle \int_{\mathbb{R}^n}|w(z)|\frac{1}{1 + |z|^{n+\sigma_0}}dz     & = & \displaystyle \int_{\mathbb{R}^n}|v(Q^{-1}z)|\frac{1}{1 + |z|^{n+\sigma_0}}dz \\
         & & \\
         & = & \displaystyle |\det(Q)| \int_{\mathbb{R}^n}|v(y)|\frac{1}{1 + |Qy|^{n+\sigma_0}}dy\\
         & & \\
         & = & \displaystyle |\det(Q)|\int_{\mathbb{R}^n}|v(y)|\mathcal{W}_{\phi_0}(y) dy,
    \end{array}
$$
that is
$$
    \left\|w\right\|_{L^1\left(\mathbb{R}^n, \frac{1}{1 + |.|^{n+\sigma_0}}\right)} = \det(Q) \left\|v\right\|_{L^1 \left( \mathbb{R}^n, \mathcal{W}_{\phi_0} \right)}.
$$
We then obtain that $w \in \mathcal{C}(\overline{B}_1) \cap L^1\left(\mathbb{R}^n, \frac{1}{1 + |.|^{n+\sigma_0}}\right)$.
It is straightforward to check that the function $w$ solves in the viscosity sense
  $$
         \int_{\mathbb{R}^n}\delta(w,-,z)\frac{\overline{b}(y)}{|y|^{\frac{n+\sigma}{2}}} dy = 0 \quad \mbox{in $B_1$},
    $$
where $\overline{b}(y) = b(Q^{-1}y)$. By \cite[Theorem 4.1]{kriventsov2013c}, there exists $\alpha_*$ such that
$$
    \left\|w\right\|_{\mathcal{C}^{1,\alpha_*}(B_{1/2})} \leq C\left(\left\|w\right\|_{L^1\left(\mathbb{R}^n, \frac{1}{1 + |.|^{n+\sigma_0}}\right)} + \left\|w\right\|_{L^\infty(B_1)}\right).
$$
Rescalling back to $v$ we obtain
$$
    \left\|v\right\|_{\mathcal{C}^{1,\alpha_*} \left( S^{\phi_0}_{1/2} \right) } \leq \overline{C} \left( \left\|v\right\|_{L^1\left( \mathbb{R}^n, \mathcal{W}_{\phi_0} \right) } + \left\|v\right\|_{L^\infty \left( S^{\phi_0}_1 \right) } \right),
$$
where $\overline{C}(C,|Q|,|Q^{-1}|)$.
\end{proof}

We are now ready to deliver the proof of the main theorem in its discrete version. 

\begin{teo}\label{RCVTEO}
   Let $\sigma \geq \sigma_0 >1$, let $\phi \in \mathcal{C}^2(\mathbb{R}^n)$ satisfy $(\ref{MONGEAMPEREEQ})$ and assumption [A1], and let $u$ be a viscosity solution to $(\ref{MAINEQ})$ in $S^\phi_{2R}$. There exists a small $\eta > 0$ and a large $R>0$ such that if
$$\sup_{y \in \mathbb{R}^n} \left| b_{\alpha \beta}(x,y) - b(y) \right| < \eta, \quad \forall \alpha \in \mathcal{A}, \forall \beta \in \mathcal{B}, \quad \forall x \in S^\phi_{2 R},$$
$$\| u \|_{L^\infty (\mathbb{R}^n)} \leq 1$$
and
$$\| f \|_{L^\infty (S^\phi_{2 R})} \leq \eta ,$$
then we can find universal constants $\theta, \lambda, C_2 > 0$ and a sequence of linear functions $l_k = a_k + b_k\cdot x$ such that
\begin{eqnarray*}
\sup_{B_{\theta \lambda^k}} |u-l_k| & \leq & \lambda^{k(1+\alpha_0)}\\
\\
|a_{k+1} - a_k| & \leq & \lambda^{k(1+\alpha_0)}\\
\\
\left| b_{k+1} - b_k \right| & \leq & C_2 \lambda^{k \alpha_0},
\end{eqnarray*}
for every $k \in \mathbb{N}$. The exponent $\alpha_0$ satisfies $\alpha_0 < \min\{\alpha_* , \alpha_1  \}$, where $\alpha_1$ is from assumption [A1] and $\alpha_*$ is from Theorem \ref{GRADIENTESTIMATESFORSOLUTIONSWITHGROWTH}.
\end{teo}

\medskip

\begin{proof}
We proceed by induction. For step $k=0$, take $l_0 = l_1 = 0$. Assume the result holds up to order $k$ and let us show it also holds for $k+1$. Define
$$
w_k(x) = \frac{1}{\lambda^{k(1+\alpha_0)}}[u-l_k](\lambda^kx), \quad  x \in \mathbb{R}^n
$$
and consider $\theta$ so small that
 $$
    B_\theta \subset S_1^{\phi_0}.
$$

Observe first that the scaled function 
$$
\phi_{\lambda^k}(y) := \lambda^{-2k}\phi(\lambda^k y), \quad y \in \mathbb{R}^n
$$
solves the same Monge-Amp\`ere equation 
$$
\det \left( D^2 \phi_{\lambda^k}(y) \right) = F(\lambda^k y) = \overline{F}(y),
$$
with the right-hand side satisfying the same bounds as $F$,  $\gamma < \overline{F} < \Gamma$, and that the level sets of $\phi_{\lambda^k}$ are related with the level sets of $\phi$ by
$$
v_{\phi_{\lambda^k}}(y) = \lambda^{-2k}v_\phi(\lambda^k y).
$$
Now $w_k$ solves an equation with the same ellipticity constants (but rescaled kernels), namely 
$$
   \inf_{\alpha \in \mathcal{A}} \sup_{\beta \in \mathcal{B}}\int_{\mathbb{R}^n}     \delta(w_k,x,y)  \frac{b_{\alpha \beta}(\lambda^k x,\lambda^ky)}{v_{\phi_{\lambda^k}}(y)^{\frac{n+\sigma}{2}}}dy = \lambda^{k(\sigma - 1 - \alpha_0)}f(\lambda^k x)
$$
for $x \in \lambda^{-k} S^\phi_{2 R} = S^{\phi_{\lambda^k}}_{2 R\lambda^{-k}}$. Since, due to \eqref{MAASSUMP2},
$$\sigma - 1 - \alpha_0 \geq \sigma_0 - 1 - \alpha_0 > \alpha_1 -\alpha_0 > 0$$ 
and $\lambda < 1$, we have 
$$\left\| \lambda^{k(\sigma - 1 - \alpha_0)}f(\lambda^k x) \right\|_{L^\infty(\lambda^{-k}S^\phi_{2 R})} \leq \| f \|_{L^\infty \left( S^\phi_{2 R} \right) } \leq \eta.$$
Therefore, $w_k$ solves
\begin{equation}\label{EQUATIONFORWK}
   -\eta <  \inf_{\alpha \in \mathcal{A}} \sup_{\beta \in \mathcal{B}}\int_{\mathbb{R}^n}     \delta(w_k,x,y)  \frac{b_{\alpha \beta}(\lambda^k x,\lambda^ky)}{v_{\phi_{\lambda^k}}(y)^{\frac{n+\sigma}{2}}}dy < \eta
\end{equation}
for $x \in \lambda^{-k} S^\phi_{2 R} = S^{\phi_{\lambda^k}}_{2 R\lambda^{-k}}$.

Let $\phi_0$ be the function from Proposition \ref{LOCALLYELLIP} and take $R$ large such that $S_1^{\phi_0} \subset S^\phi_{2 R}$. Then, since $\phi_0$ is 2-homogeneous,
$$
   S_{2 R \lambda^{-k}}^{\phi_{\lambda^k}} = \lambda^{-k}S^\phi_{2 R} \supset \lambda^{-k}S_1^{\phi_0} = S^{\phi_0}_{\lambda^{-k}} \supset S_1^{\phi_0},
$$
so that $(\ref{EQUATIONFORWK})$ holds for $S_1^{\phi_0}$ for every $k \in \mathbb{N}$. From Corollary \ref{SCALEDHOLDER}, $w_k$ is H\"older continuous in $S^{\phi_{\lambda^k}}_{R\overline{\rho}}$. Notice that, by Lemma \ref{PREM1} with $\Psi_1 = \phi_{\lambda^k}, \Psi_2 = |.|^2$ and radius $R_0$, we obtain for large $R$
$$
   B_{R_0} \subset S^{\phi_{\lambda^k}}_{R\overline{\rho}},
$$
and so $w_k$ is H\"older continuous in $B_{R_0}$, where $R_0$ is from Lemma $\ref{RA4}$. We can now apply Lemma \ref{RA4} to the function $w_k$, 
finding $h:\mathbb{R}^n \rightarrow \mathbb{R}$ satisfying $h \in \mathcal{C}(S_1^{\phi_0})$, $|h| \leq (1 + \zeta)$ in $\mathbb{R}^n$,
$$\int_{\mathbb{R}^n}\delta(h,x,y)\frac{b(y)}{v_{\phi_0}(y)^{\frac{n+\sigma}{2}}}dy = 0,$$
in the viscosity sense in $S_1^{\phi_0}$, and $|w_k - h| < \epsilon$ in $S_1^{\phi_0}$, for some small $\epsilon$ to be chosen later. By Theorem \ref{GRADIENTESTIMATESFORSOLUTIONSWITHGROWTH}, we obtain
$$
   \left\|h\right\|_{\mathcal{C}^{1,\alpha_*}(S^{\phi_0}_{1/2})} \leq C_2.
$$

Letting $\bar{l}(x) = h(0) + \nabla h(0)\cdot x$, we have, by the mean value theorem,
$$
   \begin{array}{lll}
      |h(x) - \bar{l}(x)|  & \leq &\displaystyle [\nabla h]_{C^{\alpha_\ast}(S_{1/2}^{\phi_0})}|x|^{1 + \alpha_\ast}\\
      & & \\
        &\leq & \displaystyle C_2|x|^{1 + \alpha_\ast}
   \end{array}
$$
for $x \in B_{\theta/2}$. 

Since $|w_k| \leq 1$ in $B_{\theta}$, we have $|h| \leq 1+\epsilon$ in $B_{\theta}$ and then $|h(0)| \leq 1+\epsilon $. Also, by $C^{1,\alpha_\ast}$ estimates, we have $|\nabla h(0)| \leq C_2$. Therefore, we have the following estimates
$$
\begin{array}{lll}
|w_k - \bar{l}| & \leq & |w_k - h| + |h - \bar{l}|\\
& & \\
& \leq & \epsilon(\eta,R) + C_2|x|^{1 + \alpha_\ast} \quad \mbox{in} \quad B_{\frac{\theta}{2}}\\
\\
|w_k - \bar{l}| & \leq & |w_k - h| + |h| + |\bar{l}|\\
& & \\
& \leq & 3 \epsilon(\eta,R) + 2 + C_2|x| \quad \mbox{in} \quad B_\theta \backslash B_{\frac{\theta}{2}}\\
\\
|w_k - \bar{l}| & \leq & |w_k| + |\bar{l}|\\
& & \\
& \leq & |x|^{1+\alpha_1} + |h(0)| +|\nabla h(0)| \; |x| \\
& & \\
& \leq & |x|^{1+\alpha_1} + 1 + \epsilon(\eta,R) + C_2|x| \quad \mbox{in} \quad \mathbb{R}^n \backslash B_\theta.
\end{array}
$$

Let us now fix $\lambda>0$, to be chosen later, such that $\epsilon \leq \lambda^{1 + \alpha_\ast}$. Define
$$
l_{k+1}(x) := l_k(x) + \lambda^{k(1+\alpha_0)}\bar{l}(\lambda^{-k}x),
$$
and
$$
\begin{array}{lll}
w_{k+1}(x) & = &\displaystyle  \frac{1}{\lambda^{(k+1)(1+\alpha_0)}}[u - l_{k+1}](\lambda^{k+1}x)\\
& & \\
& = & \displaystyle  \frac{1}{\lambda^{1+\alpha_0}}[w_k - \bar{l}](\lambda x).
\end{array}
$$
We have the following estimates for the scaling above:

$$
\begin{array}{lll}
|w_{k+1}(x)| & \leq & \lambda^{-1 - \alpha_0}(\epsilon + C_2\lambda^{1+\alpha_\ast}|x|^{1+\alpha_\ast}) \quad \mbox{in} \quad \lambda^{-1}B_{\frac{\theta}{2}}\\
\\
|w_{k+1}(x)| & \leq & \lambda^{-1 -\alpha_0}(3 \epsilon + 2 + C_2\lambda|x|) \quad \mbox{in} \quad \lambda^{-1}B_{\theta}\backslash \lambda^{-1}B_{\frac{\theta}{2}}\\
\\
|w_{k+1}(x)| & \leq & \lambda^{-1 -\alpha_0}(\lambda^{1+\alpha_1}|x|^{1+\alpha_1} + 1 + \epsilon + C_2\lambda| x|) \quad \mbox{in} \quad \mathbb{R}^n \backslash \lambda^{-1}B_{\theta}.
\end{array}
\bigskip
$$
\noindent Using the fact that $\epsilon \leq \lambda^{1 + \alpha_\ast}\leq 1$, we obtain

$$
\begin{array}{lll}
|w_{k+1}(x)| & \leq & \lambda^{\alpha_\ast - \alpha_0}(1 + C_2 |x|^{1+\alpha_\ast})  \quad \mbox{in} \quad \lambda^{-1}B_{\frac{\theta}{2}}\\
\\
|w_{k+1}(x)| & \leq & 5 \lambda^{-1 - \alpha_0}  + C_2\lambda^{-\alpha_0}|x| \quad \mbox{in} \quad \lambda^{-1}B_{\theta}\backslash \lambda^{-1}B_{\frac{\theta}{2}}\\
\\
|w_{k+1}(x)| & \leq & \lambda^{\alpha_1 -\alpha_0}|x|^{1+\alpha_1} + 2\lambda^{-1 -\alpha_0} + C_2\lambda^{-\alpha_0}| x| \quad \mbox{in} \quad \mathbb{R}^n \backslash \lambda^{-1}B_{\theta}.
\end{array}
$$

\medskip

\noindent Notice now that for $\lambda \leq 1/2$, we obtain
$$
   B_{\lambda^{-1}\theta/2} \supset B_{\theta}.
$$
Recalling that $\alpha_0 < \min\{\alpha_1,\alpha_* \}$, we have, by the first estimate above,
$$
\begin{array}{lll}
|w_{k+1}(x)| & \leq & \lambda^{\alpha_* - \alpha_0}(1 + C_2|x|^{1+\alpha_*})\\
& \leq & \lambda^{\alpha_* - \alpha_0}(1 + C_2)\\
& \leq & 1
\end{array}
$$
in $B_\theta$, as long as $\lambda$ is chosen such that
$$
   \lambda \leq \left(\frac{1}{1 + C_2} \right)^{\frac{1}{\alpha_* - \alpha_0}}.
$$

We will choose $\lambda$ so small that $|w_{k+1} (x)| \leq |x|^{1+\alpha_1}$ outside $B_\theta$ as well. A point $x \in \mathbb{R}^n \backslash B_\theta$ must be in one of the sets
$$
   B_{\lambda^{-1}\theta/2},\quad  B_{\lambda^{-1}\theta} \backslash  B_{\lambda^{-1}\theta/2} \quad \mbox{or} \quad \mathbb{R}^n \backslash  B_{\lambda^{-1}\theta}. 
$$
If, for instance, $x \in B_{\lambda^{-1}\theta/2}$, we have $\theta \leq |x| < \lambda^{-1}\theta/2$, and so
$$
\begin{array}{lll}
|w_{k+1}(x)| & \leq & \lambda^{\alpha_* - \alpha_0}(1 + C_2|x|^{1+\alpha_*})\\
& \leq & \lambda^{\alpha_* - \alpha_0}(\theta^{-1-\alpha_1} + C_2|x|^{\alpha_* - \alpha_1})|x|^{1+\alpha_1}\\
& \leq & 3\theta^{-1-\alpha_1}\lambda^{\alpha_1 - \alpha_0}(1+C_2)|x|^{1+\alpha_1}\\
& \leq & \lambda^{\alpha_* - \alpha_0}\theta^{-1-\alpha_1}|x|^{1+\alpha_1} + C_2\lambda^{\alpha_* - \alpha_0}|x|^{\alpha_* - \alpha_1}|x|^{1+\alpha_1}\\
& \leq & \left( \lambda^{\alpha_* - \alpha_0}\theta^{-1-\alpha_1} + C_2\theta^{\alpha_* - \alpha_1} . \right. \\
& & \left. .  \max\left\{\left(\frac{1}{2}\right)^{\alpha_* - \alpha_1}\lambda^{\alpha_1 - \alpha_0},\lambda^{\alpha_* - \alpha_0} \right\}  \right) |x|^{1+\alpha_1} \\
& \leq & |x|^{1+\alpha_1},
\end{array}
$$
as long as $\lambda$ is chosen small such that
$$
    \left(\lambda^{\alpha_* - \alpha_0}\theta^{-1-\alpha_1} + C_2\theta^{\alpha_* - \alpha_1}\max\left\{\left(\frac{1}{2}\right)^{\alpha_* - \alpha_1}\lambda^{\alpha_1 - \alpha_0},\lambda^{\alpha_* - \alpha_0} \right\}  \right) \leq 1,
$$
recall that both $\alpha_* - \alpha_0$ and $\alpha_1 - \alpha_0$ are positive and that $\alpha_* - \alpha_1$ may be positive or negative.

Now, if $x \in B_{\lambda^{-1}\theta} \backslash  B_{\lambda^{-1}\theta/2}$, we have $\lambda^{-1}\theta/2\leq |x|< \lambda^{-1}\theta $. By the second estimate, we obtain
$$
\begin{array}{lll}
|w_{k+1}(x)| & \leq & 5\lambda^{-1 -\alpha_0} + C_2\lambda^{-\alpha_0}|x| \\
& = & 5\lambda^{\alpha_1 - \alpha_0}\lambda^{-1 - \alpha_1} + C_2\lambda^{\alpha_1 - \alpha_0}\lambda^{-\alpha_1}|x|\\
& \leq & 2 \lambda^{\alpha_1 - \alpha_0} \max\{5,C_2\}\left(\frac{\theta}{2} \right)^{-1-\alpha_1}|x|^{1+\alpha_1}\\
& \leq &|x|^{1+\alpha_1},
\end{array}
$$
as long as we choose $\lambda$ so small that
$$
    2 \lambda^{\alpha_1 - \alpha_0} \max\{5,C_2\}\left(\frac{\theta}{2} \right)^{-1-\alpha_1} \leq 1.
$$
Finally, if $x \in \mathbb{R}^n \backslash  B_{\lambda^{-1}\theta}$, we have $|x| \geq \lambda^{-1}\theta$. By the third estimate, we get
$$
\begin{array}{lll}
|w_{k+1}(x)| & \leq & \lambda^{\alpha_1 - \alpha_0}|x|^{1+\alpha_1} + 2\lambda^{-1 -\alpha_0} +  C_2\lambda^{-\alpha_0}|x|\\
&\leq & \lambda^{\alpha_1 - \alpha_0}\left(1 + 2\theta^{-1-\alpha_1} + C_2\theta^{-\alpha_1} \right)   |x|^{1+\alpha_1}\\
& \leq & |x|^{1+\alpha_1},
\end{array}
$$
as long as we take $\lambda$ small such that
$$
  \lambda^{\alpha_1 - \alpha_0}\left(1 + 2\theta^{-1-\alpha_1} + C_2\theta^{-\alpha_1} \right) \leq 1. 
$$
We choose $\lambda$ such that all of the above conditions hold and we get 
$$|w_{k+1}(x)| \leq |x|^{1+\alpha_1}, \qquad x \in \mathbb{R}^n \backslash B_\theta, $$
and so $|w_{k+1}| \leq (1 + \zeta)$ in $\mathbb{R}^n$ as desired. 

Notice as well that, for $x \in B_{\lambda^{k+1}\theta}$ and $|w_{k+1}| \leq 1$ in $B_\theta$, we have
    $$
\begin{array}{lll}
|u(x) - l_{k+1}(x)| & = & \lambda^{(k+1)(1+\alpha_0)}|w_{k+1}(\lambda^{-(k+1)}x)|\\
& \leq & \lambda^{(k+1)(1 + \alpha_0)}.
\end{array}
$$
By the definition of $l_{k+1}$, we have $a_{k+1} = a_k + \lambda^{k(1+\alpha_0)}h(0)$ and $b_{k+1} = b_k + \lambda^{k\alpha_0}\nabla h(0)$ and so the theorem is proven.
\end{proof}

Now, the full regularity estimate of Theorem \ref{MAINTHM} follows as in \cite{caffarelli2011regularity}, using a covering argument.

\bigskip

{\small \noindent{\bf Acknowledgments.} D.P. partially supported by Capes-Fapitec and CNPq. A.S. supported by a Capes PhD scholarship. J.M.U. partially supported by the King Abdullah University of Science and Technology (KAUST), by Funda\c c\~ao para a Ci\^encia e a Tecnologia, through project PTDC/MAT-PUR/28686/2017 and by the Centre for Mathematics of the University of Coimbra (UIDB/00324/2020, funded by the Portuguese Government through FCT/MCTES).}

\end{document}